\documentclass[11pt,reqno]{amsproc}
\usepackage[margin=1in]{geometry}
\usepackage{amsmath, amsthm, amssymb}
\usepackage{hyperref}
\hypersetup{colorlinks=true, pdfstartview=FitV, linkcolor=blue,citecolor=blue, urlcolor=blue}
\usepackage[abbrev,lite,nobysame]{amsrefs}
\usepackage{times}
\usepackage[usenames,dvipsnames]{color}
\usepackage{bbm}
\usepackage{bm}
\usepackage{subcaption}
\usepackage{mathtools}
\usepackage{pdfpages}
\usepackage{comment}
\definecolor{labelkey}{rgb}{0,0,1}

\newcommand{\R}{\mathbb{R}}

\newcommand{\N}{\mathbb{N}}

\newcommand{\B}{\mathcal{B}}

\providecommand{\R}{\mathbb{R}}

\providecommand{\N}{\mathbb{N}}

\newcommand{\pv}{\operatorname{p.\!v.}}

\renewcommand{\leq}{\leqslant}
\renewcommand{\geq}{\geqslant}

\newcommand{\dist}{\operatorname{dist}}

  \newtheorem{thm}{Theorem}[section]
  \newtheorem{cor}[thm]{Corollary}
  \newtheorem{Lemma}[thm]{Lemma}
  
  \theoremstyle{definition}
    
  \theoremstyle{remark}
  \newtheorem{remark}[thm]{Remark}

\usepackage{etoolbox}
\patchcmd{\subsubsection}{\itshape}{\itshape\bfseries}{}{} 

\title[]{A Brezis and Peletier type result for the fractional Robin function}

\author[]{Sidy M. Djitte and Franck Sueur}

\address[S. M. Djitte]{Department of Mathematics
Maison du nombre, 6 avenue de la Fonte, 
University of Luxembourg,  
L-4364 Esch-sur-Alzette, Luxembourg}
\email{sidymoctar.djitte@uni.lu}

\address[F. Sueur]{Department of Mathematics
Maison du nombre, 6 avenue de la Fonte, 
University of Luxembourg,  
L-4364 Esch-sur-Alzette, Luxembourg}
\email{Franck.Sueur@uni.lu}

\subjclass[2025]{35C15, 47G30} 
\keywords{Pohozaev identity, Green function, Robin function}

\begin{document}


\begin{abstract}
This paper is devoted to the Laplacian operator of fractional order $s\in (0,1)$ in several dimensions. We consider the equation $(-\Delta)^su=f(x,u)$ in $\Omega$, $u=0$ in $\Omega^c$ and establish a representation formula for partial derivatives of solutions in terms of the normal derivative $u/\delta^s$. As a consequence, we prove that solutions to the overdetermined problem $(-\Delta)^su=f(x,u)$ in $\Omega$, $u=0$ in $\Omega^c$, and $u/\delta^s=0$ on $\partial\Omega$ are globally Lipschitz continuous provided that $2s>1$. We also prove a Pohozaev-type identity for the Green function and, in particular, obtain a formula for the gradient of the Robin function, which extends to the fractional setting some results obtained by Br\'ezis and Peletier in  \cite{Bresiz} in the classical case of the Laplacian. Finally, an application to the nondegeneracy of critical points of the fractional Robin function in symmetric domains is discussed.
\end{abstract}
\maketitle

\setcounter{tocdepth}{3}
\tableofcontents

\section{Introduction and main results}


\subsection{Introduction}
\label{sec-intro}

The fractional Laplace operators appear in several  disciplines of mathematics: functional analysis, PDEs, probability theory; in diverse applications issued from biology, ecology or finance. They also have numerous definitions, based on Fourier analysis, harmonic extensions, semigroup theory or quadratic form, among others, see \cites{FRO,Kwas} and the references therein for more.
One definition of  the  fractional Laplace operator of order $s$, with $s\in (0,1)$, in $N\in \N_*$ dimensions, 
is based on the following explicit formula:
$$
(-\Delta)^s \,  u (x) := c_{N,s}\,  \pv \int_{\R^N}\frac{u(x)-u(z)}{|x-z|^{N+2s}} \, dz,
$$
with 
\begin{equation}
    \label{defC}
    c_{N,s} :=\pi^{N/2}\,  s4^s\,  \frac{\Gamma(\frac{N+2s}{2})}{\Gamma(1-s)},
\end{equation}
where $\Gamma$ is the gamma function, and where $\pv$ refers to the Cauchy principal value. 
Above $u$ is a function defined on  $\R^N$ with real values. 
The reason for the presence of the normalization constant $c_{N,s}$ is to match with the Fourier definition  which sets the fractional Laplace operator $(-\Delta)^s$ as the Fourier multiplier of symbol 
$|\xi|^{2s}$. 
On the other hand, the operator  $(-\Delta)^s$ is naturally associated with the bilinear form: 
$$
\mathcal E_s(u,v):=\frac{c_{N,s}}{2}\iint_{\R^N\times\R^N}\frac{(u(y)-u(z))(v(y)-v(z))}{|y-z|^{N+2s}}dydz.
$$

Throughout this manuscript, we shall denote by $H^s(\R^N)$ the set of all those $L^2(\R^N)$ functions $u$ for which $\mathcal E_s(u,u)<\infty$ and for a given bounded open set $\Omega\subset\R^N$, we let $\mathcal H^s_0(\Omega)$  be the space of all the elements of $ H^s(\R^N)$ for which $u\equiv 0$ in $\R^N\setminus\Omega$. In what follows, unless otherwise stated, $\Omega$ will be a bounded open set in $\R^N$ of class $C^{1,1}$ with $N\in \N_*$.

\subsection{Representation formula for partial derivatives}
Let $f: \Omega\to\R$ be a bounded, given source term.  The Dirichlet problem for the fractional Laplacian reads:
\begin{equation}\label{Dir-01}
(-\Delta)^s u=f\;\;\text{in} \quad\Omega\quad\&\quad 
u=0\quad 
\textrm{in }\;\;\R^N\setminus\Omega. 
\end{equation}
It is well known that the problem above admits a unique bounded solution which can be represented by 
\begin{equation}\label{repres-sol-Dir}
u(x)=\int_{\Omega}G_s(x,z)f(z)dz\;\;\;\text{for}\;x\in\Omega.
\end{equation}
Formula \eqref{repres-sol-Dir} is rather classical in potential theory, see the works \cite{BB1, BB2,FRO,Kwas}
 and the references therein. In the above, $G_s(x,\cdot)$ stands for the Green function associated with the operator $(-\Delta)^s$, that is, the solution to \begin{equation} \label{DIR}
    (-\Delta)^sG_s(x,\cdot)=\delta_x \quad\text{in}\quad \mathcal D'(\Omega)\qquad\&\qquad G_s(x,\cdot)=0\quad\text{in}\quad \R^N\setminus\Omega.
\end{equation}
Here $\delta_x$ denotes the Dirac delta distribution, and $\mathcal D'(\Omega)$ is the space of distributions on $\Omega$.
Formula \eqref{repres-sol-Dir} plays an important role in the qualitative and quantitative analysis of the problem \eqref{Dir-01}. For example, if $f\in L^\infty(\Omega)$, using the estimate on the Green function --see e.g \eqref{eq:greenfunct-estimate}-- one easily derives from \eqref{repres-sol-Dir} that $u$ is then also in $L^\infty(\Omega)$  with $\|u\|_{L^\infty(\Omega)}\leq c\|f\|_{L^\infty(\Omega)}$.
Further applications of the formula \eqref{repres-sol-Dir} arise in regularity theory, but also in the study of symmetry and sign properties of solutions (see \cite{FallTobias}). A generalization of formula  \eqref{repres-sol-Dir} for solutions with nonzero Dirichlet data can be found in \cite{Abatangelo}. One may also consult \cite{ASS} and the references therein for representation formulas for solutions to higher-order fractional Dirichlet problems.  
\par\,

Here we are interested in representation formulas for partial derivatives of solutions in terms of the source term $f$. 
Because of the singularity of the  Green function, one cannot simply apply a derivative to the formula \eqref{repres-sol-Dir} as one would for a regular integral with a parameter. In the classical case, and when the data are sufficiently regular, a representation formula for partial derivatives of solutions can be derived by a simple application  of Green's third formula, which we recall in what follows. Given $u\in C^2(\Omega)\cap C^1(\overline\Omega)$, we have 
\begin{equation*}\label{classical-repres-formula}
    u(x)=\int_{\Omega}G_1(x,z)(-\Delta)u(z)dz-\int_{\partial\Omega}\gamma_N( G_1(x,\cdot))(z)u(z)d\sigma,\;\;\;\text{for\,\, $x\in\Omega$},
\end{equation*}
where $G_1(x,\cdot)$ is the Green function of the Laplacian with singularity at $x$ and $\gamma_N(G_1(x,\cdot)(z) = \nabla_yG_1(x,z)\cdot \nu (z)$ for $z\in \partial \Omega$ is the Neumann trace.
\par\;

Indeed, if $u$ is sufficiently regular, say, for instance, $u\in C^3(\Omega)\cap C^2(\overline\Omega)$ so that $\partial_i u\in C^2(\Omega)\cap C^1(\overline \Omega)$, then plugging  $\partial_iu$ into the identity above and integrating by parts the first term, we get 
\begin{equation*}
    \partial_iu(x)= - \int_{\Omega}\partial_{ z_i} G_1(x,z)(-\Delta)u(z)dz-\int_{\partial\Omega}\gamma_N(G_1(x,\cdot))(z)\partial_i u(z)d\sigma,\;\;\;\text{for\,\, $x\in\Omega$}.
\end{equation*}
Thus, if $f\in C^1(\overline\Omega)$, $\Omega$ is smooth enough, and $u$ is the solution of 
$
(-\Delta)u=f$ in $\Omega$ and  $u=0$ on $\partial\Omega$, then $\partial_iu$ can be represented by 
\begin{equation}\label{repres-deriv-classical-Dir}
    \partial_iu(x)=-\int_{\Omega}\partial_{z_i} G_1(x,z)f(z)dz
   -\int_{\partial\Omega}\gamma_N (G_1(x,\cdot))(z)\partial_i u(z)d\sigma,\;\;\;\text{for\,\, $x\in\Omega$}.
\end{equation}
\par\;

Our first result establishes the counterpart of the identity \eqref{repres-deriv-classical-Dir} when the Laplacian is replaced by the fractional Laplacian $(-\Delta)^s$.  To state the result, we let $f:\R^N \times \R \to \R,\; (h,q) \mapsto f(h,q)$ with 
\begin{equation}\label{H}
\tag{H}
f \in 
\begin{cases}
C^\beta_{\mathrm{loc}}(\R^N \times \R), & \text{if } 2s>1,\\[2mm]
C^{1,\beta}_{\mathrm{loc}}(\R^N \times \R), & \text{if } 2s\leq 1,
\end{cases}
\text{ for some } \beta \in (0,1).
\end{equation}
Recall that a function $u\in C^s(\R^N)\cap\mathcal H^s_0(\Omega)$ is a weak solution of
\begin{equation}\label{sect3-H}
(-\Delta)^su=f(x,u)\;\;\text{in $\Omega$}\quad\&\quad u=0\;\;\text{in}\;\;\R^N\setminus\Omega,
\end{equation}
if 
$$\mathcal{E}_s(u, \phi) =\int_{\Omega}f(z, u(z))\phi(z)dz, $$
for all $\phi\in C^\infty_c(\Omega)$. Under the above assumptions on $f$ and by regularity theory (see \cite{LS, RS}) we know that any weak solution belongs in $C^1_{loc}(\Omega)$. We denote  $\partial_{x_i}u= \partial u/\partial {x_i}$ for $i=1,2,\cdots,N$, and $\nu_i=\nu\cdot e_i$ the $i$-coordinate of the outward unit normal  $\nu$ to $\partial\Omega$. For the sake of simplicity, for a sufficiently regular function $w$, we denote by $\gamma_0^s(w):= (w/\delta^s)\big|_{\partial\Omega}$ the fractional Neumann trace; here $\delta:=\dist(\cdot,\partial\Omega)$  is the distance function to the boundary $\partial\Omega$. With these notations, we have:
\\
\begin{thm}\label{main-result}
Let $N> 2s$ and let $\Omega$ be a bounded open set of $\R^N$ of class $C^{1,1}$ and $f:\R^N\times\R\to \R$ satisfies \eqref{H}. Let $u$ be a weak solution of \eqref{sect3-H}. Then, for all $x\in \Omega$, we have 
\begin{equation}\label{repres-partial-deriv}
\partial_{x_i} u(x)= -\Gamma^2(1+s)\int_{\partial \Omega} \gamma_0^s(u)\gamma_0^s(G_s(x,\cdot)) \nu_i\,d\sigma -\int_\Omega \partial_{z_i} G_s(x,z)f(z,u(z))\,dz,\quad \text{if}\,\, 2s>1.
\end{equation} 
\par\;
In the case $2s\leq 1$, we have
\begin{equation}\label{repres-partial-deriv-1}
\partial_{x_i} u(x)=  -\Gamma^2(1+s)\int_{\partial \Omega} \gamma_0^s(u)\gamma_0^s(G_s(x,\cdot))\nu_i\,d\sigma+\int_\Omega  G_s(x,\cdot)\Big({\partial_{h_i} f(z,u(z))+\partial_q f(z,u(z))\partial_{z_i} u(z)\Big)}dz .
\end{equation}
In particular, in the case $(-\Delta)^su=\text{const}$ in $\Omega$ \;\;\&\;\; $u=0$ in $\R^N\setminus\Omega$, we have 
\begin{equation} \label{dedu}
    \partial_{x_i} u(x)= - \Gamma^2(1+s)\int_{\partial \Omega} \gamma_0^s(u)\gamma_0^s(G_s(x,\cdot)) \nu_i\,d\sigma\quad\text{for all $s\in (0,1)$}.
\end{equation}
\end{thm}
\par\;

As a simple consequence of Theorem \ref{main-result} we have 

\begin{cor}\label{cor}
    Let $f$ be as above, and $u\in C^s(\R^N)\cap \mathcal H^s_0(\Omega)$ be a solution to the problem
\begin{equation*}
\left\{ \begin{array}{rcll} (-\Delta)^su &=& f(x,u(x)) &\textrm{in }\;\;\Omega, \\ u&=&0&
\textrm{in }\;\;\Omega^c \end{array}\right. .
\end{equation*}
Assume moreover that $(u/\delta^s)\big|_{\partial\Omega}=0$ . Then $\nabla u\in L^\infty(\Omega)$ provided that $2s>1$.
\end{cor}
Corollary \ref{cor} readily follows from \eqref{repres-partial-deriv} by using that $\nabla_zG_s(x,z)$ is --uniformly in $x$-- $L^1$-integrable in $\Omega$ provided that $2s>1$, see e.g \cite[Lemma 9]{Bogdan2011}.
\begin{remark}
 First, by regularity theory (see \cite{RS}) for each $x\in \Omega$, the functions: $\partial\Omega\to \R, z\mapsto \gamma_0^s(G_s(x,\cdot))(z)$, and $z\mapsto\gamma_0^s(u)(z)$ are continuous. Moreover, $\partial_{y_i} G_s(x,\cdot)\in L^1(\Omega)$ provided that $s\in (1/2,1)$ (see appendix \ref{subsec-9.1}). Finally, since $\delta^{1-s}\nabla u\in L^\infty(\Omega)$ (see e.g \cite{RS}), and $\delta^{s-1}G_s(x,\cdot)\in L^1(\Omega)$, it is not difficult to check that the second integral in \eqref{repres-partial-deriv-1} is finite.  Therefore, all the quantities appearing in the identities above are well defined.
 \end{remark}
 

%
\begin{remark}
When the data are sufficiently regular, it is possible to derive a formula like \eqref{repres-partial-deriv-1} by using the following Green's type identity for the fractional Laplacian (see \cite[Proposition 1.2.2]{Abatangelo}): for any $u\in C^{2s+\varepsilon}_{loc}(\Omega)$ such that $\delta^{1-s}u$ is continuous up to the boundary, and  

\[
\|u\|_{\mathcal L^1_s(\R^N\setminus\Omega,~ \delta^s)}:=\int_{\R^N\setminus\Omega}\delta^{-s}(z)|u(z)|(1+|z|)^{-N-2s}dz< + \infty, 
\]

with $u=0$ in $\R^N\setminus\overline{\Omega}$, we have
\begin{equation}\label{Abatangelo}
u(x)=\int_{\Omega}G_s(x,z)(-\Delta)^su(z)dz+\int_{\partial\Omega}\gamma_0^s(G_s(x,\cdot))\mathcal{E}u(\cdot)d\sigma\quad\text{for $x\in\Omega$},
\end{equation}
where 
\[
\mathcal{E}u(\sigma):=\lim_{\Omega\ni x\to\theta}\frac{u(x)}{\int_{\partial\Omega}\mathcal{M}_\Omega(x,\sigma')d\sigma'}\quad\text{and}\quad \mathcal{M}_\Omega(x,\sigma'):=\lim_{\Omega\ni y\to \sigma'}G_s(x,y)/\delta^s(y).
\]
\par\;
Indeed, if $u$ solves \eqref{sect3-H} with $f(x,u)=f(x)\in C^2(\overline{\Omega})$, then $\partial_ju\in C^{2s+\varepsilon}_{loc}(\Omega)\cap \mathcal L^1_s(\R^N\setminus\Omega,\, \delta^s)$ solves 
    
\begin{equation}\label{partial-u-pde}
\left\{ \begin{array}{rcll} (-\Delta)^s\partial_ju &=& \partial_jf &\textrm{in }\;\;\Omega , \\ \partial_ju&=&0&
\textrm{in }\;\;\R^N\setminus\overline\Omega, \\ \delta^{1-s}(x)\partial_ju(x)&=&s\gamma_0^s(u)(x)\nu_j(x)&\text{on}\;\; \partial\Omega. \end{array}\right. 
\end{equation}
The last identity in \eqref{partial-u-pde} follows from  \cite[Eq (1.5)]{FallSven}. Using \eqref{Abatangelo} with $\partial_ju$ and integrating by parts in the first integral, we get
\begin{align*}\label{Abatangelo}
\partial_ju(x)&=\int_{\Omega}G_s(x,z)(-\Delta)^s(\partial_ju)(z)dz+\int_{\partial\Omega}\gamma_0^s(G_s(x,\cdot))(\theta)\mathcal{E}(\partial_ju)(\sigma)d\sigma\\
&=\int_{\Omega}G_s(x,z)\partial_jf(z)dz+\int_{\partial\Omega}\gamma_0^s(G_s(x,\cdot))(\sigma)\mathcal{E}(\partial_ju)(\theta)d\sigma.
\end{align*}
We arrive at the desired identity once we prove that $\mathcal{E}(\partial_j u)(\sigma)=-\Gamma^2(1+s)\gamma_0^s(u)(\sigma)\nu_j(\sigma)$ holds for all $\sigma\in \partial\Omega$. Unfortunately, it is not clear how to obtain the latter for an arbitrary bounded open set $\Omega$.
\end{remark}

\textit{Idea of the proof of \ref{main-result}}. To prove Theorem \ref{main-result}, given the remark above, we use instead the following identity (see, e.g \cite[Theorem 1.3]{DFW2023} or \cite[Theorem 1.9]{RO-S}).
\begin{equation}\label{RS-DFW}
\int_\Omega \partial_{z_i} v(-\Delta)^sw\,dz =
-\int_\Omega \partial_{z_i}w(-\Delta)^sv\,dz-\Gamma^2(1+s)\int_{\partial\Omega}\gamma_0^s(v)\gamma_0^s(w)\,\nu_i\,d\sigma.
\end{equation}
Formally, we get \eqref{repres-partial-deriv} using $G_s(x,\cdot)$ and $u$ as test functions in \eqref{RS-DFW}. Unfortunately, the function $G_s(x,\cdot)$ is too irregular to be admissible in \eqref{RS-DFW}. To overcome this difficulty, we approximate $G_s(x,\cdot)$ by a $C^\infty_c$-function of the form $\xi_k \phi^\mu_{x}G_s(x,\cdot)$ where $\xi_k$  and $\phi^\mu_{x}$ are suitable cut-off functions which vanish near the boundary $\partial\Omega$ and near the singular point $x$ of the Green function, respectively. Next, we plug $\xi_k \phi^\mu_{x}G_s(x,\cdot)$ and $\xi_k u$ into \eqref{RS-DFW}. We expand by using the product rule for the fractional Laplacian, and then pass to the limit as $k\to+\infty$ and $\mu\to 0^+$ respectively to arrive at the desired identity. 
 For the passage to the limit with respect to $k$ we follow the analysis  done in \cite{DFW}, the novelty here mainly concerns the passage to the limit with respect to $\mu$. We refer to Section \ref{sec4.1} below for more details.

\subsection{Representation formulas for  gradient of the fractional Robin function}
\label{sec-rep-grad}
Let $N\geq 2s$. The fractional Green function $G_s$ can be split into:
\begin{equation}\label{Eq-spliting-of-Green}
    G_s(x,\cdot)= F_s(x,\cdot)-H_\Omega(x,\cdot),
\end{equation}
where 
\begin{equation}\label{funda}
   F_s(x,\cdot) :=\frac{b_{N,s}}{|x-\cdot|^{N-2s}}\quad\text{if $N> 2s$}\qquad\&\qquad F_s(x,\cdot)=-\frac{1}{\pi}\log|x-\cdot|\quad\text{if}\quad N=2s,
\end{equation}
with
\begin{equation} \label{fundaC}
b_{N,s} :=\pi^{-N/2} \, 4^{-s}  \, \frac{\Gamma(\frac{N-2s}{2})}{\Gamma(s)},
\end{equation}
 is the  fundamental solution of $(-\Delta)^s$ and 
$H_\Omega(x,\cdot)$ solves the equation
\begin{equation}\label{regular-part-of-Green}
\left\{ \begin{array}{rcll} (-\Delta)^s H_\Omega(x,\cdot)&=& 0  &\textrm{in }\Omega, \\ H_\Omega(x,z)&=&F_s(x,z)&
\textrm{in }\R^N\setminus\Omega. \end{array}\right. 
\end{equation}

In this section, we are interested in a representation formula for the gradient of the fractional Robin function. We recall that the fractional Robin function associated with the domain $\Omega$, denoted as $\mathcal R_s^\Omega:\Omega\to \R$ is given for $x$ in $\Omega$ by 
\begin{equation}\label{def-s-robin}
\mathcal R_s^\Omega(x):= H_\Omega(x,x),
\end{equation}
 where $H_\Omega(x,\cdot)$ is the regular part of the Green function $G_s(x,\cdot)$ as defined in \eqref{regular-part-of-Green}.
 To simplify the notation, from now on we simply write $\mathcal R_s$ in place of $\mathcal R_s^\Omega$. When $s=1$, i.e, in the classical case of the Laplacian, we write $\mathcal R$ instead of $\mathcal R_1$.  

The Robin function plays an important role in various fields of mathematics such as geometric function theory, capacity theory, and concentration problems (see e.g \cite{Bandle, Bresiz,Flucher} and the references therein). In particular, critical points of the Robin function (or the so-called harmonic centers of $\Omega$) appear when studying elliptic boundary value problems involving energy concentration \cite{Flucher}. And also for elliptic problems involving critical Sobolev exponent, the number of solutions is linked to the number of (non-degenerate) critical points of the Robin function (see e.g \cite{O.Rey}). Despite the great importance of the Robin function, many questions regarding its properties remain unanswered even in the classical case. Below, we prove two basic properties (integral representation of the gradient and non-degeneracy in symmetric domains) of the fractional Robin function. First, let us recall the following result obtained by Br\'ezis and Peletier in \cite{Bresiz} that expresses the  derivative of the classical Robin function with respect to the space variable $x$ in terms of an integral of the normal derivative of the Green function over the domain's boundary. 
 \begin{thm}\cite{Bresiz}\label{thm-brezis}
     Let $\Omega$ be a bounded open set of $\R^N$ ($N\geq 2$) and let  $\mathcal R$ be the Robin function of $\Omega$ associated with the classical Laplacian. Then  for any $x$ in $\Omega$,
     \begin{align}\label{repres-gradient-class-Robin}
         \nabla\mathcal R(x)=\int_{\partial\Omega}( \gamma_N (G_1(x,\cdot)))^2\, \nu \, d\sigma .
     \end{align}
 \end{thm}
 We recall that $\nu$ is the outward unit normal to $\partial\Omega$.
 \par\;
 
One of our main results extends Theorem \ref{thm-brezis} to the fractional setting. More generally, we have the following 
\begin{thm}
\label{Thm.2}
Let  $s\in (0,1)$, $N> 2s$ and let $\Omega$ be a bounded open set of $\R^N$ of class $C^{1,1}$.  Then for any $x,y\in \Omega$ with $x\neq y$ and for any $i=1,\cdots, N$, there holds:
\begin{equation}\label{partial-Green}
 \partial_{x_i} G_s(y,x)+\partial_{y_i} G_s(x,y)=  -\Gamma^2(1+s)\int_{\partial \Omega} 
 \gamma_0^s(G_s(x,\cdot)) \gamma_0^s(G_s(y,\cdot))
 \nu_i \, d\sigma.    
\end{equation}
\end{thm}
The proof of Theorem \ref{Thm.2} is similar to the proof of Theorem \ref{main-result}. To obtain the first, we simultaneously use the functions $\R^N\to\R, \;z\mapsto \phi^\gamma_{y}(z)G_s(y,z)\xi_k(z)$ and $\R^N\to\R,\;z\mapsto\phi^\mu_{x}(z)G_s(x,z)\xi_k(z)$ with $x\neq y$ as test functions in \eqref{RS-DFW} and then argue as in the proof of Theorem \ref{main-result}. We refer to Section \ref{sec4} for more details.
\par\,

As a corollary of Theorem \ref{Thm.2} we obtain the following result for the Robin function. 
\begin{cor} \label{nodeg}
Let  $s\in (0,1)$ and let $\Omega$ be a bounded open set of $\R^N$ of class $C^{1,1}$ with $N> 2s$.
For any $x \in \Omega$, we have
\begin{equation}\label{repres-gradient-Robin func}
    \nabla \mathcal R_s(x)=\Gamma^2(1+s)\int_{\partial \Omega} \Big(\gamma_0^s(G_s(x,\cdot))\Big)^2\nu\,d\sigma .
\end{equation}
\end{cor}
The proof of  Corollary \ref{nodeg} is given in Section \ref{sec-nodeg}. 
\begin{remark}
    Formula \eqref{repres-gradient-Robin func} is consistent with \eqref{repres-gradient-class-Robin} as the latter can be obtained by formally sending $s$ to  $1^-$ in \eqref{repres-gradient-Robin func}.
\end{remark}

As a further corollary, we obtain the following result in the case where the domain enjoys some symmetries. It is inspired by \cite{Gr} and proved,  thanks to formula \eqref{repres-gradient-Robin func},  in Section \ref{sec-robin}. 
\begin{cor}\label{non-deg-frac-Robin}
    Let $\Omega$ be a bounded open set of class $C^{1,1}$ of $\R^N$ with $N\geq 2$. 
  Assume that there is $j\in \{1,2,\cdots, N\}$ such that $\Omega$ is symmetric with respect to the hyperplane $H_j:=\big\{x\in\R^N:x_j=0\big\}$ and that $0$ is in $\Omega$. 
 Then $\partial_j  \mathcal R_s(0)=0$.
 If, moreover, there is $i\in \{1,2,\cdots, N\}$, with $i \neq j$, such that $\Omega$ is also symmetric with respect to the hyperplane $H_i$, then $\partial_{ij}  \mathcal R_s(0)=0$.
\end{cor}
\begin{remark}
In the case of the Laplace operator, it is known that the classical  Robin function of a convex domain in two dimensions is convex and  admits only one non-degenerate critical point, see \cite{CF} by Caffarelli and Friedman and \cite{G} for an alternative complex variable proof. Both  proofs use that, for simply connected domain, the classical  Robin function satisfies the Liouville equation, as a consequence of the Riemann conformal mapping theorem. 
Since such an equation is not available for the fractional  Robin function, the question of whether or not a similar convexity property holds as well in the fractional setting is an interesting open problem. 
In the particular case $N=1$, $\Omega=(-1,1)$ and $s=1/2$, direct computations show that $0$ is the only non-degenerate critical point of $\mathcal R_{1/2}$ since in this case the Robin function is simply given by $R_{1/2}(x)=\frac{1}{\pi}\log 2(1-x^2)$ (see, e.g \cite{Bucur}).\par\;
\end{remark}
%

\medskip

The remainder of the paper is organized as follows. In Section \ref{sec4.1} we present the proof of Theorem \ref{main-result}. Section \ref{sec4} is devoted to the proof of Theorem \ref{Thm.2}, and the last two sections are devoted to the proofs of Corollary \ref{nodeg} and of Corollary \ref{non-deg-frac-Robin}.

\section{Proof of Theorem \ref{main-result}}\label{sec4.1}

This section is devoted to the proof of Theorem \ref{main-result}.
In Section \ref{sec-scheme} we give the scheme of the proof of Theorem \ref{main-result}.
Two intermediate results are necessary in the course of the proof: Lemma \ref{Lem4.2} and Lemma \ref{Lem4.3}, which are proved, respectively, in Section \ref{sec-pr-l22} and in Section \ref{sec8}.

\subsection{Scheme of the proof}\label{sec-scheme}

The starting point of the proof of Theorem \ref{main-result} is the following identity:
\begin{equation}\label{eq:pohozaev.1}
\int_\Omega \partial_{z_i} v(-\Delta)^sw\,dz =
-\int_\Omega \partial_{z_i} w(-\Delta)^sv\,dz-\Gamma^2(1+s)\int_{\partial\Omega}\gamma_0^s(v)\gamma_0^s(w)\,\nu_i\,d\sigma.
\end{equation}
which has been established for enough regular functions $v$ and $w$ in \cite[Theorem 1.3]{DFW2023} and \cite[Theorem 1.9]{RO-S}.

Formally, to get the formula, the idea is to use $G_s(x,\cdot)$ and $u$ as test functions into \eqref{eq:pohozaev.1}. However,  for any $x\in \Omega$, the function $G_s(x,\cdot)$  is too irregular to be admissible in \eqref{eq:pohozaev.1} in both these results.
To overcome this difficulty, we approximate $G_s(x,\cdot)$ by a $C^\infty_c(\Omega)$-function of the form $\xi_k \phi^\mu_{x}G_s(x,\cdot)$ where $\xi_k$  and $\phi^\mu_{x}$ are suitable cut-off functions which vanish near the boundary $\partial\Omega$ and near the singular point $x$ of the Green function, respectively. 
More precisely, let $\rho\in C^\infty_c(-2,2)$ be such that $0\leq \rho\leq 1$ and $\rho=1$ in $(-1,1)$. Next, we fix $\delta\in C^{1,1}(\R^N)$ so that $\delta$ coincides with the signed distance function $\delta(x)=\dist(x,\R^N\setminus\Omega)-\dist(x,\Omega)$ near the boundary of $\Omega$. Moreover, we assume that $\delta$ is positive in $\Omega$ and negative in $\R^N\setminus\Omega$. Then, for any $k\in \mathbb N_*, \mu\in (0,1)$ and $x\in \Omega$, the cut-off functions $\xi_k$ and $\phi_x^\mu$ are respectively defined by: 
\begin{equation}\label{eta-k}
    \xi_k: \R^N\to \R,\; y\mapsto \xi_k(y)=1-\rho(k\delta(y)),
\end{equation}
%
\begin{equation}\label{psi-mu}
     \phi^\mu_{x}: \R^N\to \R,\; y\mapsto \phi^\mu_{x}(y)=1-\rho\Big(\frac{8}{\delta^2(x)}\frac{|x-y|^2}{\mu^2}\Big).
\end{equation}
Note that thanks to the normalizing constant $8/\delta^2(x)$, the function
\begin{equation}\label{rho-mu}
\rho^\mu_{x}(\cdot):=\rho\Big(\frac{8}{\delta^2(x)}\frac{|x-\cdot|^2}{\mu^2}\Big),
\end{equation}
is $C^\infty$ with compact support in $B_{\frac{\delta(x)}{2}}(x)\Subset\Omega$. This will be important to ensure that $\phi^\mu_{x}(\cdot)G_s(x,\cdot)$ satisfies the hypothesis of Lemma \ref{DFW} further below. 
\par\;
We have the following result regarding a double integral which involves the function $\rho$. 
It is used below in the proof of Lemma \ref{Lem4.2} and of Lemma \ref{lem-lim-G-x-y}. Its proof is given in Section \ref{ser-rad}.
\begin{Lemma}\label{Lem6.1}
Let $\rho\in C^\infty_c(-2,+2)$ such that $\rho\equiv 1$ in $(-1,+1)$ and define 

\[
a_{N,s}[\rho]:=b_{N,s}c_{N,s}\iint_{\R^N\times \R^N}\frac{\big(\rho(|y|^2)-\rho(|z|^2)\big)\big(|z|^{2s-N}-|y|^{2s-N}\big)}{|z-y|^{N+2s}}dydz\qquad\text{if $N> 2s$},
\]
and 
\[
a_{N,s}[\rho]:=-\frac{1}{\pi}c_{1,1/2}\iint_{\R\times \R}\frac{\big(\rho(|y|^2)-\rho(|z|^2)\big)\big(\log|y|-\log|z|\big)}{|z-y|^{2}}dydz \quad\text{if} \quad N=2s=1
\]

Then, for all $s\in (0,1)$, we have
\begin{equation}\label{sec7-15}
 a_{N,s}[\rho]=-2.
\end{equation}
\end{Lemma}
\par\;

We recall that $b_{N,s}$ is defined in \eqref{fundaC} and $c_{N,s} $  is defined in \eqref{defC}. 
Let us now recall the fractional product rule: 
 \begin{equation}\label{pl}
(-\Delta)^s (uv) =  v(-\Delta)^s u + u (-\Delta)^s v  -\mathcal I_s [u,v],
\end{equation}
where
\begin{align}\label{def-I-s}
\mathcal I_s [u,v] :=  c_{N,s} \, \pv\int_{\R^N} \frac{(u(\cdot)-u(y)) (v(\cdot)-v(y))}{|\cdot-y|^{N+2s}}dy .
\end{align}

The following lemma contains some important properties of the cut-off function $\phi^\mu_{x}$. These properties will be used repeatedly throughout this manuscript.

\begin{Lemma}\label{lem-ps-mu-x}Let $\phi^\mu_{x}$ be defined as above. Then the following properties hold true:
\begin{itemize}
    \item[$(i)$] For all $y\in\R^N$ we have 
    \begin{equation}\label{rescapsi}
    (-\Delta)^s\phi^\mu_{x}(y)=-\widetilde\mu^{-2s}(-\Delta)^s\big(\rho\circ |\cdot|^2\big)\big(\frac{y-x}{\widetilde\mu}\big)\;\;\text{where\;\;  $\widetilde\mu:=\frac{\delta(x)}{2\sqrt{2}}\mu$}.
\end{equation}
    \item[$(ii)$] Let $G_s(x,\cdot)$ be the Green with singularity at $x$. Then $$(-\Delta)^s\big(\phi^\mu_{x}(\cdot)G_s(x,\cdot)\big)\in L^\infty(\Omega).$$ 
    \item[$(iii)$] Let $\mathcal I_s[\cdot,\cdot]$ be the bilinear form defined in \eqref{def-I-s} and let $\xi_k$ be given by \eqref{eta-k}. Then,  for all $\varepsilon>0$, we have 
\begin{gather} \label{psi-mu-equ}
\lim_{\mu\to 0^+}
\lim_{k\to+\infty}\int_{\Omega}\xi_k^2(z)\partial_iu(z)\Bigg\{G_s(x,\cdot)(-\Delta)^s\phi^\mu_{x}-\mathcal I_s\big[\phi^\mu_{x},G_s(x,\cdot)\big]\Bigg\}dz \\ \nonumber   =\lim_{\mu\to 0^+}\lim_{k\to+\infty}\int_{B_\varepsilon(x)}\xi_k^2(z)\partial_i u(z)\Bigg\{G_s(x,\cdot)(-\Delta)^s\phi^\mu_{x}-\mathcal I_s\big[\phi^\mu_{x},G_s(x,\cdot)\big]\Bigg\}dz.
\end{gather}
\end{itemize}
Here we used the notation $\partial_i=\partial/\partial z_i$.
\end{Lemma}

\begin{proof}
The first item is a simple consequence of the integral definition of the operator $(-\Delta)^s$. To check the second item, we write 
\[
\phi^\mu_{x}(z)G_s(x,z)=b_{N,s} \phi^\mu_{x}(z)|x-z|^{2s-N}-H_\Omega(x,z)-\rho^\mu_{x}(z)H_\Omega(x,z),
\]
 where we recall the notation \eqref{rho-mu}. It is clear that $\phi^\mu_{x}(\cdot)|x-\cdot|^{2s-N}\in C^\infty(\R^N)\cap L^\infty(\R^N)$ and therefore $(-\Delta)^s\big(\phi^\mu_{x}(\cdot)|x-\cdot|^{2s-N}\big)\in L^\infty(\Omega)$. We also have $(-\Delta)^s H_\Omega(x,\cdot)=0$ in $\Omega$ by definition. Finally, since $\rho^\mu_{x}(\cdot)H_\Omega(x,\cdot)\in C^\infty_c(B_{\frac{\delta(x)}{2}}(x))$, one can easily check that $(-\Delta)^s\big(\rho^\mu_{x}(\cdot)H_\Omega(x,\cdot)\big)\in L^\infty(\Omega)$. To check the last item, we expand and examine, separately, the two terms corresponding to the integrals on the complement set $\Omega\cap B_\varepsilon^c(x)$.
First, 
 if $|x-z|>\varepsilon$ for some $\varepsilon>0$ fixed, then for $\mu>0$ sufficiently small, we have $\phi^\mu_{x}(z)=1$. Thus for $\mu>0$ sufficiently small, we get 
\begin{align}\label{Eq-Delta-s-psi-1}
  \frac{1}{c_{N,s}}\big|(-\Delta)^s\phi^\mu_{x}(z)\big|&=\pv\int_{\R^N}\frac{1-\phi^\mu_{x}(y)}{|z-y|^{N+2s}}dy=\int_{B_{\frac{\mu\delta(x)}{2}}(x)}\rho\Big(\big(\frac{2\sqrt{2}}{\mu\delta(x)}\big)^2|x-z|^2\Big)|z-y|^{-N-2s}dy,\nonumber\\
  &\leq \int_{B_{\frac{\varepsilon}{4}}(x)}\rho\Big(\big(\frac{2\sqrt{2}}{\mu\delta(x)}\big)^2|x-z|^2\Big)|x-y|^{-N-2s}dy,\nonumber\\
  &\leq C(\varepsilon)\int_{B_{\varepsilon/4}(x)}\rho\Big(\big(\frac{2\sqrt{2}}{\mu\delta(x)}\big)^2|x-z|^2\Big)dy\leq \widetilde C(\varepsilon).
\end{align}
 Moreover, since $\rho$ is compactly supported and bounded, then by \eqref{Eq-Delta-s-psi-1} we have $\lim_{\mu\to 0}(-\Delta)^s\phi^\mu_{x}(z)=0$ by the dominated convergence theorem. Next, since $G_s(x,\cdot)\partial_iu\in L^1(\Omega\setminus B_\varepsilon(x))$, we deduce by the dominated convergence theorem that 
\begin{align}\label{sec7-3}
    \lim_{\mu\to 0}\lim_{k\to+\infty}\int_{\Omega\setminus B_\varepsilon(x)}\xi_k^2G_s(x,\cdot)\partial_iu(-\Delta)^s\phi^\mu_{x}dz=0.
\end{align}

Similarly, if $|x-z|>\varepsilon$, then for $\mu>0$ sufficiently small we have 
\begin{align*}\label{sec7-4}
    \big|\mathcal I_s[\phi^\mu_{x},G_s(x,\cdot)](z)\big|&=\Big|\pv\int_{\R^N}\frac{(1-\phi^\mu_{x}(y))(G_s(x,z)-G_s(x,y))}{|y-z|^{N+2s}}dy\Big|\nonumber\\
    &=\Big|\int_{B_{\frac{\mu\delta(x)}{2}}(x)}\frac{\rho\Big(\big(\frac{2\sqrt{2}}{\mu\delta(x)}\big)^2|x-z|^2\Big)\Big(G_s(x,z)-G_s(x,y)\Big)}{|y-z|^{N+2s}}dy\Big|\nonumber\\
    &\leq C(\varepsilon)\int_{B_{\frac{\varepsilon}{4}}(x)}\rho\Big(\big(\frac{2\sqrt{2}}{\mu\delta(x)}\big)^2|x-z|^2\Big)\Big(1+|G_s(x,y)|\Big)dy\leq \widetilde C(\varepsilon),
\end{align*}
and $\mathcal I_s[\phi^\mu_{x},G_s(x,\cdot)](z)\to 0$ as $\mu\to 0^+$ since $\rho\in C^\infty_c(-2,+2)$ and $G_s(x,\cdot)\in L^1(\Omega)$. Finally, since $\partial_iu\in L^1(\Omega)$, we deduce by the dominated convergence theorem that
\begin{equation}\label{sec7-5}
    \lim_{\mu\to 0^+}\lim_{k\to+\infty}\int_{\Omega\setminus B_\varepsilon(x)}\xi_k^2(z)\partial_i u(z)\mathcal I_s[\phi^\mu_{x},G_s(x,\cdot)](z)dz=0.
\end{equation}
Claim \eqref{sec7-2} follows by combining  \eqref{sec7-3} and \eqref{sec7-5}.
\end{proof}

From now on, let also the integer $i$ be fixed equal to $1$ or $2$ or $\cdots $ or $N$.
 For any $k\in \mathbb N_*$, 
we use \eqref{eq:pohozaev.1} with
$\xi_k u\in C^1_c(\Omega)$ and $\xi_k\phi^\mu_{x}G_s(x,\cdot)\in C^\infty_c(\Omega)$ instead of $v$ and $w$; we arrive at the identity: 
\begin{align}\label{key-id}
     \int_{\Omega}\partial_i(\xi_k u)(-\Delta)^s\Big(\xi_k\phi^\mu_{x}G_s(x,\cdot)\Big)dz=-\int_{\Omega}\partial_i\big(\xi_k \phi^\mu_{x}G_s(x,\cdot)\big)(-\Delta)^s(\xi_ku)\,dz,
 \end{align}
 where we write $\partial_i=\frac{\partial }{\partial z_i}$.    Let us highlight that there is no boundary term since $\xi_k u$ is compactly supported. 
 On the one hand we apply the product law \eqref{pl} to 
 $v= \xi_k$ and $w=G_s(x,\cdot)\phi^\mu_{x} $ to obtain that 
 the left hand side of \eqref{key-id} also reads as
 %
 \begin{align}
     &\int_{\Omega}\partial_i(\xi_k u)(z)(-\Delta)^s\Big(\xi_k\phi^\mu_{x}G_s(x,\cdot)\Big)(z)dz\nonumber\\
     &=\int_{\Omega}\partial_i (\xi_ku)(z)\Bigg\{\phi^\mu_{x}G_s(x,\cdot)(-\Delta)^s\xi_k-\mathcal I_s\big[\xi_k, G_s(x,\cdot)\phi^\mu_{x}\big]\Bigg\}dz\nonumber\\
     &\quad+\int_{\Omega}\xi_k(z)\partial_i (\xi_ku)(z)(-\Delta)^s\Big(\phi^\mu_{x}G_s(x,\cdot)\Big)dz . \label{rhs1}
 \end{align}
 On the other hand, we  apply the product law \eqref{pl} to 
 $v=u$ and $w= \xi_k$ and rewrite the right hand side of \eqref{key-id} as
 \begin{align}
     \int_{\Omega}\partial_i\Big(\xi_k \phi^\mu_{x}G_s(x,\cdot)\Big)(-\Delta)^s(\xi_ku)\,dz\nonumber
     &=\int_{\Omega}\partial_i\Big(\xi_k \phi^\mu_{x}G_s(x,\cdot)\Big)\Big(u(-\Delta)^s\xi_k-\mathcal I_s(u,\xi_k)\Big)dz\nonumber\\
     &\quad+\int_{\Omega}\xi_k\partial_i\Big(\xi_k \phi^\mu_{x}G_s(x,\cdot)\Big)(-\Delta)^su(z)dz . \label{rhs2}
 \end{align}

Let us set for any $\mu\in(0,1)$,  for any $k\in \mathbb N*$, 
\begin{align*}
 A^k_{\mu}(x) &:= \int_{\Omega}\xi_k(z)\partial_i\big(\xi_k u\big)(z)(-\Delta)^s\Big(\phi^\mu_{x}G_s(x,\cdot)\Big)(z)dz, 
\\B^k_{\mu} (x)&:=  \int_{\Omega}\partial_i (\xi_ku)(z)\Bigg\{\phi^\mu_{x}G_s(x,\cdot)(-\Delta)^s\xi_k-\mathcal I_s\big[\xi_k, \phi^\mu_{x} G_s(x,\cdot)\big]\Bigg\}dz,
\\ C^k_{\mu}(x) &:= \int_{\Omega}\xi_k\partial_i\Big(\xi_k \phi^\mu_{x}G_s(x,\cdot)\Big)(-\Delta)^su(z)dz,
\\ D^k_{\mu}(x) &:=  \int_{\Omega}\partial_i\Big(\xi_k \phi^\mu_{x}G_s(x,\cdot)\Big)\Big\{u(-\Delta)^s\xi_k-\mathcal I_s[u,\xi_k]\Big\}dz . \end{align*}
 Therefore, combining \eqref{key-id},  \eqref{rhs1} and  \eqref{rhs2}, we arrive at 
\begin{equation} \label{all}
    A^k_{\mu}(x) + C^k_{\mu}(x) = - B^k_{\mu}(x) -D^k_{\mu}(x)  .
\end{equation}
Note that all the quantities above are well defined. The rest of the proof of Theorem \ref{main-result}  consists in passing to the limit, as $\mu\to 0^+$ and $k\to+\infty$, in all the terms of \eqref{all}.

For the first terms in these r.h.s. we apply the following result from  \cite[Proposition 2.4]{DFW} or \cite[Proposition 2.2]{DFW2023}. 

\begin{Lemma}\label{DFW}
Let $\Omega$ be a bounded open set of class $C^{1,1}$ and $v$, $w$ be such that $v\equiv 0\equiv w$ in $\R^N\setminus\Omega$. Moreover, assume that $w\in C^s(\R^N)$, $v/\delta^s, w/\delta^s\in C^\beta(\overline{\Omega})$ and 
 \begin{equation}\label{Gredient-estimate-FallJarohs}
\delta^{1-\beta}\nabla\big(v/\delta^s\big) \;\;\text{is bounded near the boundary $\partial\Omega$}\; \text{for some}\; \beta\in (0,1).
\end{equation}
 Then for all $Y\in C^{0,1}(\R^N,\R^N)$, 
 \begin{align}\label{DFW-id}
     &\lim_{k\to+\infty}\int_{\Omega}\nabla(\xi_k v)\cdot Y \Big[w(-\Delta)^s\xi_k-\mathcal I_s\big[\xi_k, w\big]\Big]dz
     =\frac{\Gamma^2(1+s)}{2}\int_{\partial\Omega}\gamma_0^s(v)\gamma_0^s(w)Y\cdot\nu d\sigma
 \end{align}
where $\nu$ is the outward unit normal to the boundary.
\end{Lemma}

We apply this result to both the couples 
$(v,w) =( u ,\phi^\mu_{x}G_s(x,\cdot) )$
and 
$(v,w) =( \phi^\mu_{x}G_s(x,\cdot) , u)$ with $Y=e_i=(0,\cdots 0,1,0,\cdots,0)$, and then pass to the limit as $\mu\to 0^+$
to obtain that %
\begin{align} \label{B}
\lim_{\mu\to 0^+}\lim_{k\to+\infty} B^k_{\mu}(x) &= \frac{\Gamma^2(1+s)}{2}\int_{\partial \Omega} \gamma_0^s(u)(z)\gamma_0^s(G_s(x, \cdot))(z) \nu_i(z)d\sigma(z) ,
    \\ \label{D}
    \lim_{\mu\to 0^+}\lim_{k\to+\infty}  D^k_{\mu}(x) &= \frac{\Gamma^2(1+s)}{2}\int_{\partial \Omega} \gamma_0^s(u)(z) \gamma_0^s(G_s(x,\cdot))(z) \nu_i(z)d\sigma(z) .
\end{align}
\par\;

We note that the couple $(u,\phi^\mu_{x}G_s(x,\cdot))$ satisfies the hypothesis of the Lemma above by regularity theory (see \cite{RS, FallSven}). Now, regarding 
 the passage to the limit of $ A^k_{\mu}(x)$ and $C^k_{\mu}(x)$ 
 we have the following pair of results which are, respectively, proved in Section \ref{sec-pr-l22} and in Section \ref{sec8}.
 \begin{Lemma}\label{Lem4.2}
For all $x\in \Omega$,  
  \begin{equation} \label{Lem4.2-id}
\lim_{\mu\to 0^+}\lim_{k\to+\infty} A^k_{\mu}(x)
=
\partial_{x_i} u(x).
\end{equation}
 \end{Lemma}
\begin{Lemma}\label{Lem4.3}
  For  all $x\in\Omega$, we have 
    \begin{align}\label{Lem4.3-id-1}
  \lim_{\mu\to 0^+}\lim_{k\to+\infty} 
 \mathcal  C_{k,\mu}(x)
  =
  \int_{\Omega}\partial _{z_i}G_s(x,\cdot) (-\Delta)^su\,dz=\int_{\Omega}\partial_{z_i} G_s(x,z) f(z,u(z))\,dz,\quad\text{if}\;\; 2s>1,
  \\\label{Lem4.3-id-2}
  \lim_{\mu\to 0^+}\lim_{k\to+\infty} 
 \mathcal  C_{k,\mu}(x)
  = -\int_\Omega  G_s(x,\cdot)\Bigg\{\partial_{h_i}f(z,u(z))+\partial_q f(z,u(z))\partial_{z_i} u(z)\Bigg\}dz,\quad\text{if}\;\; 2s\leq 1.
  \end{align}
  In the case $f\equiv \text{const}$, we have 
  \begin{align}\label{Lem4.3-id-3}
     \lim_{\mu\to 0^+}\lim_{k\to+\infty} 
 \mathcal  C_{k,\mu}(x)=0\quad\text{for all}\, s\in (0,1).
  \end{align}
\end{Lemma}

 Therefore, combining \eqref{all},  \eqref{B}, \eqref{D}, \eqref{Lem4.2-id},  \eqref{Lem4.3-id-1} and \eqref{Lem4.3-id-2}, we arrive at \eqref{repres-partial-deriv} and \eqref{repres-partial-deriv-1}. Up to the
proofs of Lemma \ref{Lem4.2} and Lemma \ref{Lem4.3}
the proof of Theorem \ref{main-result} is thus given.


\subsection{Proof of Lemma \ref{Lem6.1}}
\label{ser-rad} 
\begin{proof}
We first consider the case $N> 2s$. We start by showing that $a_{N,s}[\rho]$ is well defined. For that aim, we let $B_r$ be the Euclidean ball centered on $0$ with radius $r>0$. We decompose the integral into two parts, depending on whether both $x$ and $y$ belong to the ball $B_4$ or if only one of them does:
\begin{align}
  &\iint_{\R^N\times \R^N}\frac{\big(\rho(|y|^2)-\rho(|z|^2)\big)\big(|z|^{2s-N}-|y|^{2s-N}\big)}{|z-y|^{N+2s}}dydz\nonumber\\
  &=\iint_{B_4\times B_4}\frac{\big(\rho(|y|^2)-\rho(|z|^2)\big)\big(|z|^{2s-N}-|y|^{2s-N}\big)}{|z-y|^{N+2s}}dydz\label{Lem.3.2-1}\\ 
  &\;\;+ 2 \int_{B_2}\rho(|z|^2) \Big(\int_{\R^N\setminus B_4}\frac{|z|^{2s-N}-|y|^{2s-N}}{|z-y|^{N+2s}}dy \Big) dz\nonumber
\end{align}
Since 
$$\int_{\R^N}|z|^{2s-N}(1+|z|^{N+2s})^{-1}dz<\infty,$$
it is clear that the last integral is finite. To continue the proof, we distinguish two cases. 
\par\;
\underline{Case 1}: $2s< 1$. In this case, we estimate
\begin{align}
   &\Bigg| \iint_{B_4\times B_4}\frac{\big(\rho(|y|^2)-\rho(|z|^2)\big)\big(|z|^{2s-N}-|y|^{2s-N}\big)}{|z-y|^{N+2s}}dydz\Bigg|\nonumber\\
   &\leq C \iint_{B_4\times B_4}\frac{|z|^{2s-N}+|y|^{2s-N}}{|z-y|^{N-(1-2s)}}dydz<\infty.
\end{align}
\underline{Case 2}: $2s\geq 1$.
In this case we decompose the integral in \eqref{Lem.3.2-1} into two parts, depending on whether both $x$ and $y$ do not belong to the ball $B_1$ or if only one of them does:
\begin{align}
  &\iint_{B_4\times B_4}\frac{\big(\rho(|y|^2)-\rho(|z|^2)\big)\big(|z|^{2s-N}-|y|^{2s-N}\big)}{|z-y|^{N+2s}}dydz\nonumber\\
  &=\iint_{(B_4\setminus B_1)\times (B_4\setminus B_1)}\frac{\big(\rho(|y|^2)-\rho(|z|^2)\big)\big(|z|^{2s-N}-|y|^{2s-N}\big)}{|z-y|^{N+2s}}dydz \label{Lem.3.2-2}\\ \label{Lem.3.2-3}
  &\;\;+2\iint_{B_1\times (B_4\setminus B_1)}\frac{\big(\rho(|y|^2)-\rho(|z|^2)\big)\big(|z|^{2s-N}-|y|^{2s-N}\big)}{|z-y|^{N+2s}}dydz.
\end{align}
Since $z\mapsto|z|^{2s-N}$ is $C^\infty(B_4\setminus B_1)$, by Taylor expansion we easily check that the integral in \eqref{Lem.3.2-2} is finite.  Next, we recall the following elementary estimate: 
for any $b, c\geq 0$ and $b\neq 0$, we have 
\begin{equation}\label{elem-id}
|b^\beta-c^\beta|\leq (b^{\beta-1}/|\beta|
)\max\big(1,(c/b)^\beta\big)|b-c|.
\end{equation}
We apply this to $\beta=2s-N$, $b=|z|$ and $c=|y|$ yields
\begin{equation}\label{es-1}
  \big||z|^{2s-N}-|y|^{2s-N}\big|\leq C |y|^{2s-N-1}|z-y|\quad\text{for $y\in B_1\;\;\&\;\;z\in B_4\setminus B_1$.}  
\end{equation}
Now using \eqref{es-1} and that $\rho\in C^\infty_c(\R)$, it is straightforward to check the integral in \eqref{Lem.3.2-3} is finite provided that $2s>1$. It remains the case 
 $s=1/2$. For this we used the estimate: for $b>1$ and $\beta\in (0,1)$ there holds:  
\begin{align*}
    \frac{1-b^{1-N}}{N-1}=\int_1^b t^{-N}dt&\leq \Big(\int_1^bdt\Big)^\beta\Big(\int_1^b t^{-\frac{N}{1-\beta}}dt\Big)^{1-\beta}\nonumber\\
    &\leq \big(b-1\big)^\beta\Big(\int_1^\infty t^{-\frac{N}{1-\beta}}\Big)^{1-\beta}\nonumber\\
    &\leq \big(b-1\big)^\beta\Big(\frac{1-\beta}{N+\beta-1}\Big)^{1-\beta}.
\end{align*}
We apply this to $\frac{|z|}{|y|}$ with $y\in B_1$ and $z\in B_4\setminus B_1$ gives 
\begin{align*}
    &\Big||y|^{1-N}-|z|^{1-N}\Big|\\
    &=|y|^{1-N}\Big|1-\Big(\frac{|z|}{|y|}\Big)^{1-N}\Big|\\
    &\leq  (N-1)|y|^{1-N}\Big(\frac{|z|}{|y|}-1\Big)^\beta\Big(\frac{1-\beta}{N+\beta-1}\Big)^{1-\beta}\nonumber\\
    &\leq C|y|^{1-N-\beta}|z-y|^\beta.
\end{align*}
It follows that 
\begin{align}
&\iint_{B_1\times (B_4\setminus B_1)}\frac{\big(\rho(|y|^2)-\rho(|z|^2)\big)\big(|z|^{1-N}-|y|^{1-N}\big)}{|z-y|^{N+1}}dydz\nonumber\\
&\leq C\int_{B_1}|y|^{1-N-\beta}\int_{B_4\setminus B_1}\frac{dz}{|y-z|^{N-\beta}}dy< \infty.
\end{align}
In conclusion, we have 
\begin{align}
   \iint_{\R^N\times \R^N}\frac{\big(\rho(|y|^2)-\rho(|z|^2)\big)\big(|z|^{2s-N}-|y|^{2s-N}\big)}{|z-y|^{N+2s}}dydz<\infty.
\end{align}
Then, by Fubini's theorem,  we have that 
\begin{align}
&b_{N,s}c_{N,s}\iint_{\R^N\times \R^N}\frac{\big(\rho(|y|^2)-\rho(|z|^2)\big)\big(|z|^{2s-N}-|y|^{2s-N}\big)}{|z-y|^{N+2s}}dydz\nonumber\\
&=-2\int_{\R^N}b_{N,s}|z|^{2s-N}(-\Delta)^s(\rho\circ|\cdot|^2)dz=-2,
\end{align}
by using  an integration by parts over the full space $\R^N$,  that $b_{N,s}|z|^{2s-N}$ is the fundamental solution of $(-\Delta)^s$ in $\R^N$ and the value of $\rho(0)$. This finishes the proof for $N>2s$. 

For $N=2s$, the same argument goes through except that in here, to estimate the problematic term, which is  
\[
A:=\iint_{B_1\times (B_4\setminus B_1)}\frac{\big(\rho(|y|^2)-\rho(|z|^2)\big)\big(\log|z|-\log|y|\big)}{|z-y|^{2}}dydz,
\]
we simply write 
\begin{align*}
    A&=\iint_{B_{1/2}\times (B_4\setminus B_1)}\frac{\big(\rho(|y|^2)-\rho(|z|^2)\big)\big(\log|z|-\log|y|\big)}{|z-y|^{2}}dydz\\
    &+\iint_{(B_1\setminus B_{1/2})\times (B_4\setminus B_1)}\frac{\big(\rho(|y|^2)-\rho(|z|^2)\big)\big(\log|z|-\log|y|\big)}{|z-y|^{2}}dydz
\end{align*}
It is clear that the first integral is finite since $|z-y|>1/2$. The second integral is also finite since $z\mapsto\log|z|$ is Lipschitz away from the origin. The proof is  therefore finished.
\end{proof}
\subsection{Proof of Lemma \ref{Lem4.2}}
\label{sec-pr-l22}

We shall need the following result in the proof of Lemma \ref{Lem4.2} below.
\begin{Lemma}\label{ess-convergence-result} Let $N> 2s$, $w\in C(\Omega)$ and $x\in\Omega$ be fixed. Let $G_s(x,\cdot)$ be the Green function with singularity at $x$ and $\phi^\mu_{x}$ be defined as in \eqref{psi-mu}. Then for all $\varepsilon>0$ such that $\overline{B_{2\varepsilon}(x)}\subset\Omega$, we have 
\begin{align}
    \lim_{\mu\to 0^+}\int_{B_\varepsilon(x)}w(z)G_s(x,z)(-\Delta)^s\phi^\mu_{x}(z)dz=-w(x),\label{fl}\\
    \lim_{\mu\to 0^+}\int_{B_\varepsilon(x)}w(z)\mathcal I_s\big[G_s(x,\cdot), \phi^\mu_{x}\big](z)dz=-2w(x).\label{sl}
\end{align}
\end{Lemma}
The proof of Lemma \ref{ess-convergence-result} is given in Section \ref{secpro}.
\begin{remark}\label{rem-lem-2.6}
Up to choosing $\varepsilon>0$ sufficiently small, the assumption $w\in C(\Omega)$ in the lemma above can be reduced to
 assuming the continuity of $w$ at $x$.
\end{remark}
%
Recall that $x\in\Omega$ is given and that for any $\mu\in(0,1)$ and any $k\in \mathbb N*$, 
$$ A^k_{\mu}(x) := \int_{\Omega}\xi_k(z)\partial_i\big(\xi_k u\big)(z)(-\Delta)^s\Big[\phi^\mu_{x}G_s(x,\cdot)\Big](z)dz. $$
%
%
By Leibniz' rule, we have 
\begin{align*}
   A^k_{\mu}(x)
    &=\int_{\Omega}\xi_k^2(z)\partial_i u(z)(-\Delta)^s\Big[\phi^\mu_{x}G_s(x,\cdot)\Big](z)dz\nonumber\\
    &\quad+k\int_{\Omega}u\xi_k(z)\xi'(k\delta)(z)\partial_i\delta(z)(-\Delta)^s\Big[\phi^\mu_{x}G_s(x,\cdot)\Big](z)dz .
\end{align*}
Next, using the product rule \eqref{pl} with
 $u=\phi^\mu_{x} $, $v=G_s(x,\cdot) $, and also using $(-\Delta)^s G_s(x,\cdot)=0$ in $\Omega\setminus \overline{B_\mu(x)}$, we get
\begin{align*}
\small\text{$\int_{\Omega}\xi_k^2\partial_i u(-\Delta)^s\Big[\phi^\mu_{x}G_s(x,\cdot)\Big]dz
=\int_{\Omega}\xi_k^2\partial_i u\Bigg\{G_s(x,\cdot)(-\Delta)^s\phi^\mu_{x}-\mathcal I_s\big[\phi^\mu_{x},G_s(x,\cdot)\big]\Bigg\}dz.$}
\end{align*}
%
By combining the two identities above, we arrive at 
    \begin{align}\label{Lemid}
A^k_{\mu}(x) &= \int_{\Omega}\xi_k^2\partial_i u\Bigg\{G_s(x,\cdot)(-\Delta)^s\phi^\mu_{x}-\mathcal I_s\big(\phi^\mu_{x},G_s(x,\cdot)\big)\Bigg\}dz\nonumber\\
&\quad\quad  +k\int_{\Omega}u\xi_k(z)\xi'(k\delta)(z)\partial_i\delta(z)(-\Delta)^s\Big(\phi^\mu_{x}G_s(x,\cdot)\Big)(z)dz .
  \end{align}
The rest of the proof of Lemma \ref{Lem4.2}
consists in passing to the limit, as $\mu\to 0^+$ and $k\to+\infty$, in the two terms on the right hand side of \eqref{Lemid}. 
We start with a result regarding the double-limit of the first term, which by Lemma \ref{lem-ps-mu-x} can be reduced to the one of the same integral on an arbitrarily small ball centered at $x$. That is,
\begin{gather}
    \label{sec7-2}
    \lim_{\mu\to 0^+}
    \lim_{k\to+\infty}\int_{\Omega}\xi_k^2(z)\partial_iu(z)\Bigg\{G_s(x,\cdot)(-\Delta)^s\phi^\mu_{x}-\mathcal I_s\big[\phi^\mu_{x},G_s(x,\cdot)\big]\Bigg\}dz 
    \\ \nonumber   =\lim_{\mu\to 0^+}\lim_{k\to+\infty}\int_{B_\varepsilon(x)}\xi_k^2(z)\partial_i u(z)\Bigg\{G_s(x,\cdot)(-\Delta)^s\phi^\mu_{x}-\mathcal I_s\big[\phi^\mu_{x},G_s(x,\cdot)\big]\Bigg\}dz.
    \end{gather} 
\par\;
In view of \eqref{Lemid} and \eqref{sec7-2} and applying Lemma \ref{ess-convergence-result} with $w=\partial_iu\in C(\Omega)$, we get 
\begin{align*}
    &\lim_{\mu \to 0^+}\lim_{k\to+\infty}A^k_{\mu}(x)\nonumber\\
    &=\lim_{\mu\to 0^+}\lim_{k\to+\infty}\int_{B_\varepsilon(x)}\xi_k^2(z)\partial_i u(z)\Bigg\{G_s(x,\cdot)(-\Delta)^s\phi^\mu_{x}-\mathcal I_s\big[\phi^\mu_{x},G_s(x,\cdot)\big]\Bigg\}dz\nonumber\\
    &\quad\quad  +\lim_{\mu\to 0^+}\lim_{k\to+\infty}k\int_{\Omega}u\xi_k(z)\xi'(k\delta)(z)\partial_i\delta(z)(-\Delta)^s\Big(\phi^\mu_{x}G_s(x,\cdot)\Big)dz\nonumber\\
    &=\lim_{\mu\to 0^+}\int_{B_\varepsilon(x)}\partial_i u(z)\Bigg\{G_s(x,\cdot)(-\Delta)^s\phi^\mu_{x}-\mathcal I_s\big[\phi^\mu_{x},G_s(x,\cdot)\big]\Bigg\}dz\nonumber\\
    &\quad\quad  +\lim_{\mu\to 0^+}\lim_{k\to+\infty}k\int_{\Omega}u\xi_k(z)\xi'(k\delta)(z)\partial_i\delta(z)(-\Delta)^s\Big(\phi^\mu_{x}G_s(x,\cdot)\Big)dz\nonumber\\
    &\qquad=\partial_iu(x)+\lim_{\mu\to 0^+}\lim_{k\to+\infty}k\int_{\Omega}u\xi_k(z)\xi'(k\delta)(z)\partial_i\delta(z)(-\Delta)^s\Big(\phi^\mu_{x}G_s(x,\cdot)\Big)dz.
\end{align*}
Thus, to deduce Lemma \ref{Lem4.2} it suffices to  prove that
\begin{align}\label{pf}
   \lim_{\mu\to 0^+}\lim_{k\to+\infty}k\int_{\Omega}u\xi_k(z)\xi'(k\delta)(z)\partial_i\delta(z)(-\Delta)^s\Big(\phi^\mu_{x}G_s(x,\cdot)\Big)dz=0. 
\end{align}
\par\;
To see \eqref{pf}, we note that since $(-\Delta)^s\Big[\phi^\mu_{x}G_s(x,\cdot)\Big]\in L^\infty(\Omega)$ (see Lemma \ref{lem-ps-mu-x}) and $|u(z)|\leq C\delta^s(z)$ we have 
\begin{align}\label{sec7-19}
    &\Bigg|k\int_{\Omega}u\xi_k(z)\xi'(k\delta)(z)\partial_i\delta(z)(-\Delta)^s\Big[\phi^\mu_{x}G_s(x,\cdot)\Big](z)dz\Bigg|\nonumber\\
    &\leq C(\mu)k\int_{\Omega_{\frac{2}{k}}\setminus\Omega_{\frac{1}{k}}}|\rho'(k\delta(z))|\delta^s(z)dz\nonumber\\
    &\leq 2C(\mu)\int_{\Omega_{\frac{2}{k}}\setminus\Omega_{\frac{1}{k}}}\delta^{s-1}(z)|\rho'(k\delta(z))|dz\to 0\;\;\text{as}\;\;k\to+\infty
\end{align}
by the dominated convergence theorem. This ends the proof of \eqref{pf}. In the third line, we used that if $z\in \Omega_{\frac{2}{k}}\setminus\Omega_{\frac{1}{k}}$ then $1\leq k\delta(z)\leq 2$. The proof of Lemma \ref{Lem4.2} is therefore finished.

\subsection{Proof of Lemma \ref{ess-convergence-result}}
\label{secpro}
In this subsection, we prove Lemma \ref{ess-convergence-result}.
We start with the proof of \eqref{fl}. 
Using the scaling law \eqref{rescapsi} and  changing variables $\overline{z}=\frac{x-z}{\widetilde\mu}$, we obtain:
\begin{align}\label{sec7-6}
   &-\int_{B_\varepsilon(x)}w(z)G_s(x,z)(-\Delta)^s\phi^\mu_{x}(z)dz,\nonumber\\
   &=\widetilde\mu^{-2s}\int_{B_\varepsilon(0)}G_s(x,x-z)w(x-z)(-\Delta)^s\big(\rho\circ |\cdot|^2\big)\big(\frac{z}{\widetilde\mu}\big)dz,\nonumber\\
   &=\int_{B_{\frac{\varepsilon}{\widetilde\mu}}(0)}\widetilde\mu^{N-2s}G_s(x,x-\widetilde\mu z)w(x-\widetilde\mu z)(-\Delta)^s\big(\rho\circ |\cdot|^2\big)(z)dz.
\end{align}

 Since $w\in L^\infty_{loc}(\Omega)$, we have $\big|1_{B_{\frac{\varepsilon}{\widetilde\mu}}(0)}(z)w(x-\widetilde\mu z)\big|\leq C$.  Moreover, by the standard estimates on the Green function, see e.g \eqref{eq:greenfunct-estimate}, one has $\big|\widetilde\mu^{N-2s}G_s(x,x-\widetilde\mu z)\big|\leq C|z|^{2s-N}$. Now, since $w$ is continuous at $x$ and 
$$ 
\int_{\R^N}|z|^{2s-N}(-\Delta)^s\big(\rho\circ |\cdot|^2\big)(z)dz<\infty,
$$
we have by the Lebesgue dominated convergence theorem that
\begin{align}\label{sec7-9}
&\lim_{ \mu\to 0^+}\int_{B_{\frac{\varepsilon}{\widetilde\mu}}(0)}\widetilde\mu^{N-2s}G_s(x,x-\widetilde\mu z)w(x-\widetilde\mu z)(-\Delta)^s\big(\rho\circ |\cdot|^2\big)dz\nonumber\\&=w(x)\int_{\R^N}b_{N,s}|z|^{2s-N}(-\Delta)^s\big(\rho\circ |\cdot|^2\big)dz.
\end{align}
Here we used: 
$\widetilde\mu^{N-2s}G_s(x,x-\widetilde\mu z)=\frac{b_{N,s}}{|z|^{N-2s}}-\widetilde\mu^{N-2s}H_s(x,x-\widetilde\mu z)\to b_{N,s}|z|^{2s-N}$ as $\mu\to 0^+$ since $H_s(x,\cdot)\in C(\Omega)$ and $N>2s$.  
Finally, as we have already seen above, since $b_{N,s}|z|^{2s-N}$ is the fundamental solution of $(-\Delta)^s$ in $\R^N$ and $\rho\circ|\cdot|^2\in C^\infty_c(\R^N)$, we deduce by an integration by parts over the entire space $\R^N$ that:
\begin{equation}\label{sec7-8}
\int_{\R^N}b_{N,s}|z|^{2s-N}(-\Delta)^s\big(\rho\circ |\cdot|^2\big)dz=\rho(0)=1.
\end{equation}
Combining \eqref{sec7-8} and \eqref{sec7-9} we obtain \eqref{fl}. The proof of \eqref{sl} is somehow similar to the proof of \eqref{fl} with some minor differences, but for completeness, we give the full details of the argument. 

We first observe that
\begin{align}\label{sec7-11}
    &\lim_{\mu\to 0^+}\lim_{k\to+\infty}\int_{B_\varepsilon(x)}w(z)\mathcal I_s\big[\phi^\mu_{x},G_s(x,\cdot)\big](z)dz\nonumber\\
    &=\lim_{\mu\to 0^+}c_{N,s}\int_{B_\varepsilon(x)}w(z)\int_{B_{2\varepsilon}(x)}\frac{(\phi^\mu_{x}(y)-\phi^\mu_{x}(z))(G_s(x,y)-G_s(x,z))}{|y-z|^{N+2s}}dydz.
\end{align}
Indeed, for $\mu\in(0,1)$ sufficiently small and $|x-y|>\varepsilon$, we have $\phi^\mu_{x}(y)=0$ and thus recalling \eqref{eq:greenfunct-estimate}, we have 
\begin{align}
    &\Bigg|\int_{B_\varepsilon(x)}w(z)\int_{\R^N\setminus B_{2\varepsilon}(x)}\frac{(\phi^\mu_{x}(y)-\phi^\mu_{x}(z))(G_s(x,y)-G_s(x,z))}{|y-z|^{N+2s}}dydz\Bigg|,\nonumber\\
    &\leq C_1(\varepsilon,\Omega, N,s)\int_{B_\varepsilon(x)}\big(1+|G_s(x,z)|\big)\big|w(z)\big|\rho\Big(\big(\frac{2\sqrt{2}}{\mu\delta(x)}\big)^2|x-z|^2\Big)\int_{\R^N\setminus B_{2\varepsilon}(x)}\frac{dy}{1+|y|^{N+2s}}dz,\nonumber\\
    &\leq C_2(\varepsilon,\Omega, N,s)\int_{B_\varepsilon(x)}\big(1+|G_s(x,z)|\big)\rho\Big(\big(\frac{2\sqrt{2}}{\mu\delta(x)}\big)^2|x-z|^2\Big)dz.\label{Eq9/7}
\end{align}
The r.h.s of \eqref{Eq9/7} converges to zero as $\mu\to 0^+$ which proves the claim \eqref{sec7-11}.
\par\;

Next, applying the change of variables: $\overline{z}=\frac{x-z}{\widetilde\mu}$ and $\overline{y}=\frac{x-y}{\widetilde\mu}$, we get 
\begin{align}\label{sec7-12}
  &\int_{B_\varepsilon(x)}w(z)\int_{B_{2\varepsilon}(x)}\frac{(\phi^{\widetilde\mu}_{x}(y)-\phi^{\widetilde\mu}_{x}(z))(G_s(x,y)-G_s(x,z))}{|y-z|^{N+2s}}dydz \\ \label{issy2}
  &= \int_{B_{\frac{\varepsilon}{\widetilde\mu}(0)}  }w(x-\widetilde\mu z)\int_{B_{\frac{2\varepsilon}{\widetilde\mu}}(0)}\frac{\Big(\rho(|z|^2)-\rho(|y|^2)\Big)\Big(\widetilde\mu^{N-2s}G_s(x,x-\widetilde\mu y)-\widetilde\mu^{N-2s}G_s(x,x-\widetilde\mu z)\Big)}{|y-z|^{N+2s}}dydz.
\end{align}
We are going to split the integral on the right hand side above into two parts according to the decomposition:
\begin{align*}
   &\widetilde\mu^{N-2s}G_s(x,x-\widetilde\mu y)-\widetilde\mu^{N-2s}G_s(x,x-\widetilde\mu z)\\
   &=b_{N,s}\big(|y|^{2s-N}-|z|^{2s-N}\big)-\widetilde\mu^{N-2s}\big(H_s(x,x-\widetilde\mu y)-H_s(x,x-\widetilde\mu z)\big).
\end{align*}
Since $H_s(x,\cdot)$ is $s$-harmonic, we know $\nabla_yH_s(x,y)$ is locally bounded in $\Omega$. Consequently 
$$\big|H_s(x,x-\widetilde\mu y)-H_s(x,x-\widetilde\mu z)\big|\leq C(x,\varepsilon)\widetilde\mu|y-z|\quad \text{for all}\,z\in B_{\frac{\varepsilon}{\widetilde\mu}}(0)\;\;\&\;\;\text{ for all}\;\;y\in B_{\frac{2\varepsilon}{\widetilde\mu}}(0) .$$
Thus, we have
\begin{align}
&\Bigg|\iint_{B_{\frac{\varepsilon}{\widetilde\mu}(0)}\times B_{\frac{2\varepsilon}{\widetilde\mu}(0)}}\frac{\Big(\rho(|z|^2)-\rho(|y|^2)\Big)\Big(\widetilde\mu^{N-2s}H_s(x,x-\widetilde\mu y)-\widetilde\mu^{N-2s}H_s(x,x-\widetilde\mu z)\Big)}{|y-z|^{N+2s}}dydz\Bigg|\nonumber\\  \label{sec7-13}
&\leq C(x)\widetilde\mu^{N-2s+1}\iint_{B_{\frac{\varepsilon}{\widetilde\mu}(0)}\times B_{\frac{2\varepsilon}{\widetilde\mu}(0)}}\frac{\big|\rho(|z|^2)-\rho(|y|^2)\big||z-y|}{|y-z|^{N+2s}}dydz\;\;\to\;\; 0\;\;\text{as}\;\; \mu\to 0^+.
\end{align}
On the other hand, by Lemma \ref{Lem6.1} and the Lebesgue dominated convergence theorem, 
\begin{align}\label{sec7-14}
&\lim_{\mu\to 0^+}\iint_{B_{\frac{\varepsilon}{\widetilde\mu}(0)}\times B_{\frac{2\varepsilon}{\mu}(0)}}w(x-\widetilde\mu z)\frac{\Big(\rho(|z|^2)-\rho(|y|^2)\Big)\Big(b_{N,s}|y|^{2s-N}-b_{N,s}|z|^{2s-N}\Big)}{|y-z|^{N+2s}}dydz  \nonumber\\
&= w(x)b_{N,s}\iint_{\R^N\times\R^N}\frac{\Big(\rho(|z|^2)-\rho(|y|^2)\Big)\Big(|y|^{2s-N}-|z|^{2s-N}\Big)}{|y-z|^{N+2s}}dydz =- \frac{ 2 }{c_{N,s }} w(x).
\end{align}
Thus, combining \eqref{sec7-13} and \eqref{sec7-14} we get
\begin{align}
\label{issy}
  &\small\text{$\lim_{\mu\to 0^+}c_{N,s}\int_{B_{\frac{\varepsilon}{\widetilde\mu}(0)}}w(x-\widetilde\mu z)\int_{B_{\frac{2\varepsilon}{\widetilde\mu}}}\frac{\Big(\rho(|z|^2)-\rho(|y|^2)\Big)\Big(\widetilde\mu^{N-2s}G_s(x,x-\widetilde\mu y)-\widetilde\mu^{N-2s}G_s(x,x-\widetilde\mu z)\Big)}{|y-z|^{N+2s}}dydz$}\nonumber\\
  &\quad\quad\quad\quad=-2w(x).
\end{align}
Combining 
\eqref{sec7-11},  \eqref{issy2} and 
\eqref{issy} we arrive at  \eqref{sl}.

\subsection{Proof of Lemma \ref{Lem4.3}}
\label{sec8}

 Let us recall that 
\begin{align}
\mathcal  C_{k,\mu}(x):= \int_{\Omega}\xi_k\partial_i\Big(\xi_k \phi^\mu_{x}G_s(x,\cdot)\Big)(-\Delta)^su(z)dz\quad\text{for $x\in\Omega$}.
 \end{align}
 We consider the two different cases separately.
 \par\;

 \underline{Case 1}: $2s>1$. By Leibniz rule, we have 
\begin{align}\label{sec8-1}
\mathcal C_{k,\mu}(x)&=\int_{\Omega}\xi_k^2(z)\phi^\mu_{x}(z)\partial_i G_s(x,z)(-\Delta)^su(z)\,dz\nonumber\\
 &\;\;+\int_{\Omega}\xi_k(z)G_s(x,z)\partial_i\big(\xi_k\phi^\mu_{x}\big)(z)(-\Delta)^su(z)dz
\end{align}
Using that $(-\Delta)^su\in L^\infty(\Omega)$ and recalling \eqref{eq:greenfunct-estimate}, we estimate 
\begin{align}\label{sec8-2}
&\lim_{\mu\to 0^+} \lim_{k\to+\infty}\Big|\int_{\Omega}\xi_k^2(z)\partial_i\phi^\mu_{x}G_s(x,z)(-\Delta)^su(z)dz\Big|\nonumber\\
&\leq C_1(x)\lim_{\mu\to 0^+}\int_{B\big(x, \frac{\delta(x)}{2}\mu\big)\setminus B\big(x, \frac{\delta(x)}{2\sqrt{2}}\mu\big)}\Bigg|\frac{(x-z)\cdot e_i}{\mu^2}\rho'\Big(\big(\frac{2\sqrt{2}}{\mu\delta(x)}\big)^2|x-z|^2\Big)\Bigg||x-z|^{2s-N}dz\nonumber\\
&\leq C_2(x) \lim_{\mu\to 0^+}\int_{B\big(x, \frac{\delta(x)}{2}\mu\big)\setminus B\big(x, \frac{\delta(x)}{2\sqrt{2}}\mu\big)}  \frac{1}{\mu}\Bigg|\rho'\Big(\big(\frac{2\sqrt{2}}{\mu\delta(x)}\big)^2|x-z|^2\Big)\Bigg||x-z|^{2s-N}dz\nonumber\\
&\leq C_3(x)\lim_{\mu\to 0^+} \mu^{2s-1}
\int_{ \frac{\delta(x)}{2\sqrt{2}} }^{\frac{\delta(x)}{2}} 
r^{2s-1} |\rho'(r^2)|dr=0 ,
\end{align}
since $2s>1$. A similar argument as in \eqref{pf} gives 
\begin{align}\label{sec8-3}
\lim_{\mu\to 0^+}\lim_{k\to+\infty}\int_{\Omega}\xi_k(z)\phi^\mu_{x}\partial_i\xi_k(z)G_s(x,z)(-\Delta)^su(z)dz=0.
\end{align}
The first identity in the Lemma follows by taking the limit in \eqref{sec8-1} and using \eqref{sec8-2} and \eqref{sec8-3}.
\par\;

\underline{Case 2}: $2s\leq 1$. In this case, we have, by an integration by parts, that
\begin{align}\label{cases-fl}
\mathcal C_{k,\mu}(x):&= \int_{\Omega}\xi_k\partial_i\Big(\xi_k \phi^\mu_{x}G_s(x,\cdot)\Big)(-\Delta)^su(z)dz\nonumber\\
 &= \int_{\Omega}\partial_i\Big(\xi_k \phi^\mu_{x}G_s(x,\cdot)\Big)(z)\xi_k(z)f(z,u(z))dz\nonumber\\
 &=-\int_{\Omega}\xi_k \phi^\mu_{x}G_s(x,\cdot) \partial_i\big(\xi_kf(\cdot, u(\cdot))\big)(z)dz\nonumber\\
 &=-\int_{\Omega}\xi_k^2\phi^\mu_{x}G_s(x,\cdot)\Bigg\{\frac{\partial f}{\partial h_i}(z,u(z))+\partial_q f(z,u(z))\partial_{y_i} u(z)\Bigg\}dz\nonumber\\
 &\quad\quad-k\int_{\Omega}\xi_k(z)\phi^\mu_{x}(z)G_s(x,z)\rho'(k\delta(z))\partial_i\delta(z) f(z,u(z))dz
 \end{align}
On the one hand, an argument similar to the proof of \eqref{pf} gives 
\begin{align}\label{case2-sl}
    \lim_{k\to+\infty}k\int_{\Omega}\xi_k(z)\phi^\mu_{x}(z)G_s(x,z)\rho'(k\delta(z))\partial_i\delta(z) f(z,u(z))dz=0.
\end{align}
On the other hand, by the Lebesgue dominated convergence theorem we have:
\begin{align}\label{case2-tl}
 &\lim_{\mu\to 0^+}\lim_{k\to+\infty}\int_{\Omega}\xi_k^2\phi^\mu_{x}G_s(x,\cdot)\Bigg\{\partial_{h_i} f(z,u(z))+\partial_q f(z,u(z))\frac{\partial u}{\partial y_i}(z)\Bigg\}dz\nonumber\\
 &\quad\quad \quad \quad=  \int_{\Omega}G_s(x,z)\Bigg\{\partial_{h_i} f(z,u(z))+\partial_q f(z,u(z))\frac{\partial u}{\partial y_i}(z)\Bigg\}dz.
\end{align}
In view of \eqref{case2-sl} and \eqref{case2-tl}, we get the desired identity by taking the limit as $k\to+\infty$ and as $\mu\to 0^+$ in \eqref{cases-fl}.  Finally, to see \eqref{Lem4.3-id-3} we use the Leibniz' rule to write
\begin{align*}
 \lim_{\mu\to 0^+}\lim_{k\to+\infty} \mathcal C_{k,\mu}(x)&=\lim_{\mu\to 0^+}\lim_{k\to+\infty} \int_{\Omega}\xi_k\partial_i\Big(\xi_k \phi^\mu_{x}G_s(x,\cdot)\Big)(-\Delta)^su(z)dz\nonumber\\
 &=c\lim_{\mu\to 0^+} \int_{\Omega}\partial_i\Big(\phi^\mu_{x}G_s(x,\cdot)\Big)+c\lim_{\mu\to 0^+}\lim_{k\to+\infty}k\int_{\Omega}\xi_k\rho'(k\delta)\partial_i\delta\phi^\mu_{x}G_s(x,\cdot)dz.
 \end{align*}
Since $\phi^\mu_{x}G_s(x,\cdot)=0$ on $\partial\Omega$, it is clear that the first limit in the above identity is zero by the divergence theorem. We also know from the above that the second limit is zero as well. Consequently, we have $ \lim_{\mu\to 0^+}\lim_{k\to+\infty}\mathcal C_{k,\mu}(x)=0$ which proves \eqref{Lem4.3-id-3}.
The proof of Lemma \ref{Lem4.3} is therefore finished.


\section{Proof of Theorem \ref{Thm.2}}\label{sec4}

This section is devoted to the proof of Theorem  \ref{Thm.2}. 
We start with the proof of  \eqref{partial-Green}, for which an intermediate result is necessary: Lemma \ref{lem-lim-G-x-y}. Its proof is postponed to Section \ref{sec9}.

\subsection{Proof of  \eqref{partial-Green}}

First, we use  the functions $\R^N\to\R, \;z\mapsto \phi^\gamma_{y}(z)G_s(y,z)\xi_k(z)$ and $\R^N\to\R,\;z\mapsto\phi^\mu_{x}(z)G_s(x,z)\xi_k(z)$ with $x\neq y$ as test functions in \eqref{eq:pohozaev.1} and expand to obtain that 
\begin{equation} \label{all-2}
  E_k^{\mu,\gamma}(x,y) +  G_k^{\mu,\gamma}(x,y) = - F_k^{\mu,\gamma}(x,y) -  H_k^{\mu,\gamma}(x,y), 
\end{equation}
with 
\begin{align*}
E_k^{\mu,\gamma}(x,y) &:= \int_{\Omega}\xi_k(z)\partial_i\big(\xi_k \phi^\gamma_{y}G_s(y,\cdot)\big)(z)(-\Delta)^s\Big[\phi^\mu_{x}G_s(x,\cdot)\Big](z)dz, 
\\ F_k^{\mu,\gamma} (x,y)&:=  \int_{\Omega}\partial_i (\xi_k\phi^\mu_{x}G_s(x,\cdot))(z)\Big\{\phi^\gamma_{y}G_s(y,\cdot)(-\Delta)^s\xi_k-\mathcal I_s\big[\xi_k\phi^\gamma_{y}, G_s(y,\cdot)\big]\Big\}dz ,
\\ G_k^{\mu,\gamma}(x,y) &:= \int_{\Omega}\xi_k\partial_i\Big(\xi_k \phi^\mu_{x}G_s(x,\cdot)\Big)(-\Delta)^s\Big[\phi^\gamma_{y}G_s(y,\cdot)\Big](z)dz ,
\\ H_k^{\mu,\gamma}(x,y) &:=  \int_{\Omega}\partial_i\Big(\xi_k \phi^\mu_{x}G_s(x,\cdot)\Big)\Big\{\phi^\gamma_{y}G_s(y,\cdot)(-\Delta)^s\xi_k-\mathcal I_s[\phi^\mu_{x}G_s(x,\cdot),\xi_k]\Big\}dz .
 \end{align*}
On the one hand, we applied the lemma \ref{DFW} to the couple $\big(\phi^\gamma_{y}G_s(y,\cdot),\phi^\mu_{x}G_s(x,\cdot)\big)$ and then passed to the limit, first as $\mu\to 0^+$ and then as $\gamma\to 0^+$ to get that
\begin{align}\label{lim-F-x-y}
    \lim_{\gamma\to 0^+}\lim_{\mu\to 0^+}\lim_{k\to+\infty} F_k^{\mu,\gamma} (x,y)=\frac{\Gamma^2(1+s)}{2}\int_{\partial \Omega}\gamma_0^s(G_s(x,\cdot))(z)\gamma_0^s(G_s(y,\cdot))(z) \nu_i(z)d\sigma(z),
\end{align}
\begin{align}\label{lim-H-x-y}
   \lim_{\gamma\to 0^+} \lim_{\mu\to 0^+}\lim_{k\to+\infty} H_k^{\mu,\gamma} (x,y)=\frac{\Gamma^2(1+s)}{2}\int_{\partial \Omega} \gamma_0^s(G_s(x,\cdot))(z) \gamma_0^s(G_s(y,\cdot))(z) \nu_i(z)d\sigma(z)
\end{align}
where $\nu$ is the outward unit normal to the boundary. On the other hand, Lemma \ref{Lem4.2} applied with $u=\phi^\gamma_{y}G_s(y,\cdot)$ yields
\begin{align*}
    \lim_{\mu\to 0^+}\lim_{k\to+\infty} E_k^{\mu,\gamma}(x,y) &=\partial_{x_i}(\phi^\gamma_{y}G_s(y,\cdot))(x)\nonumber\\
    &=\phi^\gamma_{y}(x)\partial_{x_i} G_s(y,x)-\frac{16}{\delta^2(x)}\frac{(x-y)\cdot e_i}{\gamma^2}\rho'\Big(\big(\frac{2\sqrt{2}}{\gamma\delta(x)}\big)^2|x-y|^2\Big).
\end{align*}
Passing to the limit as $\gamma\to 0^+$, and observing that 
$\rho'\Big(\big(\frac{2\sqrt{2}}{\gamma\delta(x)}\big)^2|x-y|^2\Big)$ vanishes  when $\gamma>0$ is sufficiently small, 
we get
\begin{align}\label{lim-E-x-y}
    \lim_{\gamma\to 0^+} \lim_{\mu\to 0^+}\lim_{k\to+\infty}E_k^{\mu,\gamma}(x,y)=\partial_{x_i} G_s(y,x).
\end{align}

In view of \eqref{all-2}, \eqref{lim-F-x-y}, \eqref{lim-H-x-y}, \eqref{lim-E-x-y}, the formula announced in \eqref{partial-Green} follows once we prove the following result.
\begin{Lemma}\label{lem-lim-G-x-y}
Let $s\in (0,1)$. Then for any $x,y\in\Omega$ with $x\neq y$, 
there holds: 
\begin{equation}
\lim_{\gamma\to 0^+}  \lim_{\mu\to 0^+}\lim_{k\to+\infty} G_k^{\mu, \gamma}(x,y) = \partial_{y_i} G_s(x,y).
\end{equation}
\end{Lemma}

\subsection{Proof of Lemma \ref{lem-lim-G-x-y}}\label{sec9}

First we note that since $(-\Delta)^s G_s(y,\cdot)=0$ in $\Omega\setminus\overline{B_\varepsilon(y)}$ for all $\varepsilon>0$, and we  apply the fractional product law \eqref{pl} to obtain:
\begin{align}\label{G-k-mu-gamma}
G_k^{\mu,\gamma}(x,y) = \int_{\Omega}\xi_k\partial_i\Big(\xi_k \phi^\mu_{x}G_s(x,\cdot)\Big) \Bigg\{G_s(y,\cdot)(-\Delta)^s\phi^\gamma_{y}-\mathcal I_s\big[G_s(y,\cdot),\phi^\gamma_{y}\big]\Bigg\}dz.
\end{align}
Let $x,y\in\Omega$ such that $x\neq y$ and let $G_s(x,\cdot)$ and $G_s(y,\cdot)$ be the Green functions with singularity $x$ and $y$ respectively. Recall 
\[
\phi^\gamma_{y}(z)=1-\rho\Big(\big(\frac{2\sqrt{2}}{\gamma\delta(y)}\big)^2|y-z|^2\Big) \;\;\;\text{and}\;\;\; \phi^\mu_{x}(z)=1-\rho\Big(\big(\frac{2\sqrt{2}}{\mu\delta(x)}\big)^2|x-z|^2\Big)\quad\text{for \;$\gamma,\mu\in (0,1)$.}
\]
 By Leibniz rule, we have 
 \begin{equation}\label{dedede}
G_k^{\mu,\gamma}(x,y) =  G_{k,1}^{\mu,\gamma}(x,y)+ G_{k,2}^{\mu,\gamma}(x,y) , 
 \end{equation}
 where 
\begin{align*} G_{k,1}^{\mu,\gamma}(x,y) &:= \int_{\Omega}\xi_k^2\partial_i\Big(\phi^\mu_{x}G_s(x,\cdot)\Big) \Bigg\{G_s(y,\cdot)(-\Delta)^s\phi^\gamma_{y}-\mathcal I_s\big[G_s(y,\cdot),\phi^\gamma_{y}\big]\Bigg\}dz , \nonumber\\
 G^{k,2}_{\mu,\gamma}(x,y) :=&2k\int_{\Omega}\xi_k(z)\rho'(k\delta)\partial_i\delta(z)\Big(\phi^\mu_{x}G_s(x,\cdot)\Big) (-\Delta)^s\big(\phi^\gamma_{y}G_s(z,\cdot)\big)dz .
\end{align*}
Next, arguing as in the proof of \eqref{psi-mu-equ} in Lemma \ref{lem-ps-mu-x} we have:
\begin{align}\label{sec9-1}
    &\lim_{\gamma\to 0^+}\lim_{\mu\to 0^+}\lim_{k\to +\infty}G_{k,1}^{\mu,\gamma}(x,y)\nonumber\\
    &=\small\text{$\lim_{\gamma\to 0^+}\lim_{\mu\to 0^+}\lim_{k\to+\infty}\int_{B_\varepsilon(y)}\xi_k^2(z)\partial_i\Big(\phi^\mu_{x}G_s(x,\cdot)\Big)(z) \Bigg\{G_s(y,\cdot)(-\Delta)^s\phi^\gamma_{y}-\mathcal I_s\big[G_s(y,\cdot),\phi^\gamma_{y}\big]\Bigg\}dz,$}
\end{align}
for all $\varepsilon>0$.  To proceed further, we pick
\begin{equation}\label{Eq-choice-of-vareps}
\varepsilon=\varepsilon_0:=\min\Big(\frac{\delta(y)}{4},\frac{|x-y|}{2}\Big)>0\;\;\; \text{so that}\;\;\;\overline{B_{2\varepsilon_0}(y)}\subset \Omega.
\end{equation}
\par\;
By the choice of $\varepsilon_0$, we know that if $z\in B_{\varepsilon_0}(y)$, then $|x-z|\geq 2^{-1}|x-y|>0$ and hence $|G_s(x,z)|\leq C(\varepsilon_0)$ for all $z\in B_{\varepsilon_0}(y)$. Consequently
\begin{align}\label{sec9-7}
 &\Bigg|\int_{B_{\varepsilon_0}(y)}G_s(x,z)\partial_i\phi^\mu_{x}(z) \Bigg\{G_s(y,\cdot)(-\Delta)^s\phi^\gamma_{y}-\mathcal I_s\big[G_s(y,\cdot),\phi^\gamma_{y}\big]\Bigg\}dz\Bigg|\nonumber\\
 &\leq C(\varepsilon_0, \Omega,\gamma)\int_{B_{\varepsilon_0}(y)}|z-y|^{2s-N}\frac{1}{\mu}\Big|\rho'\Big(\big(\frac{2\sqrt{2}}{\mu\delta(x)}\big)^2|x-z|^2\Big)\Big|dz\to 0\quad\text{as}\quad\mu\to 0^+ .
\end{align}
Indeed, since by the choice of $\varepsilon_0>0$ we have $|x-z|\geq 2^{-1}|x-y|>0$ for all $z\in B_{\varepsilon_0}(y)$, therefore $\rho'\Big(\big(\frac{2\sqrt{2}}{\mu\delta(x)}\big)^2|x-z|^2\Big)=0$ when $\mu>0$ is sufficiently small.  In \eqref{sec9-7} we used that $(-\Delta)^s\phi^\gamma_{y}\in L^\infty(\R^N)$ and $\mathcal I_s\big[G_s(y,\cdot),\phi^\gamma_{y}\big]\in L^\infty(B_{\varepsilon_0}(y))$. In view of \eqref{sec9-1} and \eqref{sec9-7}, we have 
    \begin{align}
    &\lim_{\gamma\to 0^+}\lim_{\mu\to 0^+}\lim_{k\to +\infty} G_{k,1}^{\mu,\gamma}(x,y)\nonumber\\
    &=\small\text{$\lim_{\gamma\to 0^+}\lim_{\mu\to 0^+}\lim_{k\to+\infty}\int_{B_{\varepsilon_0}(y)}\xi_k^2(z)\partial_i\Big(\phi^\mu_{x}G_s(x,\cdot)\Big)(z) \Bigg\{G_s(y,\cdot)(-\Delta)^s\phi^\gamma_{y}-\mathcal I_s\big[G_s(y,\cdot),\phi^\gamma_{y}\big]\Bigg\}dz$}\nonumber\\
    &=\lim_{\gamma\to 0^+}\lim_{\mu\to 0^+}\int_{B_{\varepsilon_0}(y)}\partial_i\Big(\phi^\mu_{x}G_s(x,\cdot)\Big)(z) \Bigg\{G_s(y,\cdot)(-\Delta)^s\phi^\gamma_{y}-\mathcal I_s\big[G_s(y,\cdot),\phi^\gamma_{y}\big]\Bigg\}dz\nonumber\\
    &=\lim_{\gamma\to 0^+}\int_{B_{\varepsilon_0}(y)}\partial_iG_s(x,\cdot)(z) \Bigg\{G_s(y,\cdot)(-\Delta)^s\phi^\gamma_{y}-\mathcal I_s\big[G_s(y,\cdot),\phi^\gamma_{y}\big]\Bigg\}dz. \label{findeli}
\end{align}
Using the scaling law \eqref{rescapsi} (with $\gamma$ instead of $\mu$),  we get 
\begin{align}\label{sec9-8}
   &-\int_{B_{\varepsilon_0}(y)}G_s(y,z)\partial_{z_i} G_s(x,z) (-\Delta)^s\phi^\gamma_{y}(z)dz,\nonumber\\
   &=\widetilde\gamma^{-2s}\int_{B_{\varepsilon_0}(0)}G_s(y,y-z)\partial_{z_i}  G_s(x,y-z)(-\Delta)^s\big(\rho\circ |\cdot|^2\big)\big(\frac{z}{\widetilde\gamma}\big)dz,\nonumber\\
   &=\widetilde\gamma^{N-2s}\int_{B_{\frac{\varepsilon_0}{\widetilde\gamma}}(0)}G_s(y,y-\widetilde\gamma z)\partial_{z_i}G_s(x,y-\widetilde\gamma z)(-\Delta)^s\big(\rho\circ |\cdot|^2\big)(z)dz,\nonumber\\
   &=\int_{B_{\frac{\varepsilon_0}{\widetilde\gamma}}(0)}\widetilde\gamma^{N-2s}G_s(y,y-\widetilde\gamma z)\Big[\partial_{z_i} G_s(x,y-\widetilde\gamma z)-\partial  _{z_i}G_s(x,y)\Big](-\Delta)^s\big(\rho\circ |\cdot|^2\big)(z)dz,\nonumber\\
   &\;\;\;+\frac{\partial  G_s}{\partial y_i}(x,y)\int_{B_{\frac{\varepsilon_0}{\widetilde\gamma}}(0)}\widetilde\gamma^{N-2s}G_s(y,y-\gamma z)(-\Delta)^s\big(\rho\circ |\cdot|^2\big)(z)dz.
\end{align}
By the choice of $\varepsilon_0$ and since $G_s(x,\cdot)\in C^\infty_{loc}(\Omega\setminus\{x\})$, we have 
$$
\Big|\partial_{z_i}  G_s(x,y-\widetilde\gamma z)-\partial_{z_i}  G_s(x,y)\Big|\leq \widetilde\gamma C(\varepsilon_0)|z|\quad\forall\,z\in B_{\frac{\varepsilon_0}{\widetilde\gamma}}(0).
$$
Using this and that $\big|\widetilde\gamma^{N-2s}G_s(y,y-\gamma z)\big|\leq C|z|^{2s-N}$, we get 
\begin{align}\label{sec9-9}
   &\Bigg|\int_{B_{\frac{\varepsilon_0}{\widetilde\gamma}}(0)}\widetilde\gamma^{N-2s}G_s(y,y-\widetilde\gamma z)\Big[\partial_{z_i} G_s(x,y-\widetilde\gamma z)-\partial_{z_i}  G_s(x,y)\Big](-\Delta)^s\big(\rho\circ |\cdot|^2\big)(z)dz\Bigg|\nonumber\\
   &\leq C(\varepsilon_0)\widetilde\gamma\int_{\R^N}\frac{dz}{|z|^{N-2s-1}(1+|z|^{N+2s})}\to 0\;\;\text{as}\;\; \gamma\to 0^+.
\end{align}
On the one hand, we know by \eqref{sec7-8} that
\begin{equation}\label{sec9-10}
   \lim_{\gamma\to 0^+}\int_{B_{\frac{\varepsilon_0}{\widetilde\gamma}}(0)}\widetilde\gamma^{N-2s}G_s(y,y-\widetilde\gamma z)(-\Delta)^s\big(\rho\circ |\cdot|^2\big)(z)dz= 1.
\end{equation}
Passing into the limit in \eqref{sec9-8} and taking into account \eqref{sec9-9} and \eqref{sec9-10}, we end up with 
\begin{equation}\label{sec9-11}
 \lim_{\gamma\to 0^+}\int_{B_{\varepsilon_0}(y)}\partial_{z_i}  G_s(x,z)G_s(y,z)(-\Delta)^s\phi^\gamma_{y}(z)dz =-\partial_{z_i} G_s(x,y).   
\end{equation}
A similar argument, as in the proof of \eqref{sl} above, also yields
\begin{align}\label{sec9-12}
 &\lim_{\gamma\to 0^+}\int_{\Omega}\partial_{z_i}  G_s(x,z)\mathcal I_s\big[G_s(y,z),\phi^\gamma_{y}\big](z)dz\nonumber\\
 &=\lim_{\gamma\to 0^+}\int_{B_{\varepsilon_0}(y)}\partial_{z_i}G_s(x,z)\mathcal I_s\big[G_s(y,z),\phi^\gamma_{y}\big](z)dz\nonumber\\
 &=-2\partial_{z_i} G_s(x,y).   
\end{align}

In view of \eqref{dedede}, \eqref{findeli}, \eqref{sec9-11} and \eqref{sec9-12}, we are done once we know that 
\begin{align*}
  &\lim_{\gamma\to 0^+}\lim_{\mu\to 0^+}\lim_{k\to+\infty}  G^{k,2}_{\mu,\gamma}(x,y)\\
 &=\lim_{\gamma\to 0^+}\lim_{\mu\to 0^+}\lim_{k\to+\infty} 2k\int_{\Omega}\xi_k(z)\rho'(k\delta)\partial_i\delta(z)\Big(\phi^\mu_{x}G_s(x,\cdot)\Big) (-\Delta)^s\big(\phi^\gamma_{y}G_s(z,\cdot)\big)dz\\
 &=0.
\end{align*}
The later follows by a similar argument as in \eqref{pf}. The proof of Lemma \ref{lem-lim-G-x-y} is therefore finished.

\section{Proof of Corollary \ref{nodeg}}
\label{sec-nodeg}

This section is devoted to the proof of Corollary \ref{nodeg}. Assuming $N> 2s$, we note that
$$\partial_{x_i}(|x-y|^{2s-N})+\partial_{y_i}(|x-y|^{2s-N})=0.$$ 
Now recalling the decomposition \eqref{Eq-spliting-of-Green}, the identity \eqref{partial-Green} becomes: for any $(x,y)$ in $\Omega\times\Omega$, with $x\neq y$,
$$
\partial_{x_i} H_s(y,x)+\partial_{y_i} H_s(x,y)= - \Gamma^2(1+s)\int_{\partial \Omega} \gamma_0^s(G_s(x,\cdot))(z)\gamma_0^s(G_s(y,\cdot))(z) \nu_i(z)d\sigma(z). 
$$
But now that we have removed the singular part, we can let $y$ go to $x$, and since 
$$  \partial_{x_i} \mathcal R_s(x) = \partial_{x_i} H_s(x,x)+\partial_{y_i} H_s(x,x),$$
we arrive at \eqref{repres-gradient-Robin func}. This proves the identity in the case $N> 2s$. The case $N=2s=1$ follows similarly by using again the decomposition \eqref{Eq-spliting-of-Green} and that \[
\partial_{x_j}(\log|x-y|)+\partial_{y_j}(\log|x-y|)=0.
\]

\section{Proof of Corollary \ref{non-deg-frac-Robin}}
\label{sec-robin}

This section is devoted to the proof of Corollary \ref{non-deg-frac-Robin}.
The proof is an adaptation of the proof given in \cite{G} on the corresponding problem for the classical case of the Laplacian. For $j\in \{1,2,\cdots, N\}$, 
 we  set
$$H_j:=\big\{x\in\R^N:x_j=0\big\} \quad \text{ and } \quad 
\Omega_j := H_j \cap \Omega,$$ 
 and we assume that  $\Omega$ is symmetric with respect to the hyperplane $H_j$ and that $0$ is in $\Omega$. 
We call  $T_j:\R^N\to\R^N$  the reflection with respect to the hyperplane $H_j$.
We start with the following simple lemma regarding the symmetry of the fractional Green function.
\begin{Lemma}
    Fix $j\in\{1,2,\cdots, N\}$ and let $ y\in \Omega_j$. Then we have 

    \begin{equation}\label{lab-claim}
    G_s( y, T_j(x))=G_s( y, x)\quad \text{for a.e $x\in \R^N$}.
    \end{equation}
\end{Lemma}
\begin{proof}
    We have, on the one hand, that
    \begin{align}\label{lab1}
        \int_{\Omega}G_s( y, z)(-\Delta)^s\phi(z)dz=\phi( y)\quad\forall\, \phi\in C^\infty_c(\Omega).
    \end{align}
    On the other hand, 
   using the symmetry of  $\Omega$ and the definition of fractional Laplacian, we have that for any $\phi\in C^\infty_c(\Omega)$, for any $z\in \Omega$, 
   \[
   (-\Delta)^s(\phi\circ T_j)(z)=((-\Delta)^s\phi)(T_j(z)
   \]
   By changing variables and using \eqref{lab1}  with $\phi\circ T_j$ instead of $\phi$, 
   we deduce that 
    \begin{align}\label{lab2}
        \int_{\Omega}G_s( y, T_j(z))(-\Delta)^s\phi(z)dz =\phi( y) .
    \end{align}
   where we recall that $T_j(y)=y$.  Comparing \eqref{lab1} and \eqref{lab2}, and by a density argument, we get the result.
\end{proof}
\par\;
Now, for all $ y\in \Omega_j$, by Theorem \ref{Thm.2}, 
\begin{align}
 \partial_{y_j}\mathcal R_s( y)  &=\Gamma^2(1+s)\int_{\partial\Omega}\Big(\gamma_0^s(G_s( y,\cdot))\Big)^2(\sigma)\nu_j(\sigma)d\sigma \label{lab-fin}    .
\end{align}
Using the symmetry of  $\Omega$, \eqref{lab-claim} and the fact that the normal vector satisfies 
 $\nu_j(T_j(\sigma))=-\nu_j(\sigma)$, we deduce that 
 $\partial_{y_j}\mathcal R_s( y) =0$.
  If, moreover, there is $i\in \{1,2,\cdots, N\}$, with $i \neq j$, such that $\Omega$ is also symmetric with respect to the hyperplane $H_i$, then one can differentiate the previous identity in the direction $i$ so that $\partial_{ij}  \mathcal R_s( y)=0$.

\section{Appendix}
\label{subsec-9.1}

The following properties of the Green function were used throughout the manuscript. 

Let $\Omega$ be a bounded open set of $\R^N$ of class $C^{1,1}$. Then for all $x,y\in\Omega$ with $x\neq y$, we have (see \cites{Tedeuz, CS})
\begin{equation}\label{eq:greenfunct-estimate}
    c_1\min\Big(\frac{1}{|x-y|^{N-2s}}, \frac{\delta^s(x)\delta^s(y)}{|x-y|^N}\Big)\leq \frac{G_s(x,y)}{\gamma_{N,s}}\leq c_2\min\Big(\frac{1}{|x-y|^{N-2s}}, \frac{\delta^s(x)\delta^s(y)}{|x-y|^N}\Big)
\end{equation}
for some constant $c_1, c_2>0$ and an explicit constant $\gamma_{N,s}$.  Moreover, for any arbitrary bounded open subset $\Omega$  of $\R^N$ , we have (see e.g \cite[Corollary 3.3]{Bogdan2002}):
\begin{equation}\label{gradient-estimate-green-func}
    \big|\nabla _y G_s(x,y)\big|\leq N\frac{G_s(x,y)}{\min\{|x-y|,\delta(y)\}}, \quad \text{for all $x,y\in \Omega$ with $x\neq y$}.
\end{equation}
It follows from \eqref{eq:greenfunct-estimate} and \eqref{gradient-estimate-green-func} that if $\Omega$ is of class $C^{1,1}$ and $x\in\Omega$, then $\nabla _y G_s(x,y)$ is --uniformly in $x$--  $L^1$-integrable in $\Omega$  provided that $s\in (1/2,1)$, see e.g \cite[Lemma 9]{Bogdan2011}.
\bigskip 

\section*{Declarations}
 \ \par \noindent {\bf Ethical Approval:}  Not applicable.

 \ \par \noindent {\bf Funding
:}  The first author was partially supported by the Alexander von Humboldt-Professorship program and by
 the Transregio 154 Project “Mathematical Modelling, Simulation and Optimization Using the Example
 of Gas Networks” of the Deutsche Forschungsgemeinschaft.

\ \par \noindent {\bf Availability of data and materials
:}  Not applicable.

\bigskip 
\bigskip \ \par \noindent {\bf Acknowledgements.}  This work was initiated when F. S. was visiting the  FAU DCN-AvH during some snowy days. He warmly thanks Enrique Zuazua and his team  for their kind hospitality. The first author thanks Sven Jarohs, M. Moustapha Fall and T. Weth for useful comments.


\end{document}